\documentclass[smallextended]{svjour3}

\smartqed
\usepackage{amssymb,amsxtra}
\usepackage{dsfont,latexsym,enumerate}
\usepackage{url}
\usepackage{mathrsfs}

\usepackage{color}
\definecolor{citation}{rgb}{0.2,0.6,0.2}
\definecolor{formula}{rgb}{0.1,0.2,0.6}
\definecolor{url}{rgb}{0,0,0.4}
\usepackage[colorlinks=true,linkcolor=formula,urlcolor=url,citecolor=citation]{hyperref}

\newlength{\defbaselineskip}
\usepackage{amsthm}

\numberwithin{equation}{section}
\allowdisplaybreaks[4]

\def\vs{\vspace{1mm}}
\def\dd{d_{\textrm{o}}}

\allowdisplaybreaks
\makeatletter
\DeclareRobustCommand*{\bfseries}{%
	\not@math@alphabet\bfseries\mathbf
	\fontseries\bfdefault\selectfont
	\boldmath
}

\def \e{{\varepsilon}}
\def \p{{\varphi}}
\def \l{{\mathcal{L}}}
\def \r{{\mathds{R}}}
\def \h{{\mathds{H}}}
\def \c{c}
\def\Xint#1{\mathchoice
	{\XXint\displaystyle\textstyle{#1}}%
	{\XXint\textstyle\scriptstyle{#1}}%
	{\XXint\scriptstyle\scriptscriptstyle{#1}}%
	{\XXint\scriptscriptstyle\scriptscriptstyle{#1}} %
	\!\int}
\def\XXint#1#2#3{{\setbox0=\hbox{$#1{#2#3}{\int}$}
		\vcenter{\hbox{$#2#3$}}\kern-.5\wd0}}

\def\dashint{\Xint-}

\def\vs{\vspace{1mm}}
\def\dxieta{\,{\rm d}\xi{\rm d}\eta}
\def\deta{\,{\rm d}\eta}
\def\dxi{\,{\rm d}\xi}
\def\dd{d_{\textrm{o}}}
\allowdisplaybreaks
\makeatletter
\DeclareRobustCommand*{\bfseries}{%
	\not@math@alphabet\bfseries\mathbf
	\fontseries\bfdefault\selectfont
	\boldmath
}

\DeclareMathOperator*{\osc}{osc}
\DeclareMathOperator*{\dist}{dist}
\DeclareMathOperator*{\essinf}{ess \, \inf}
\DeclareMathOperator*{\esssup}{ess \, \sup}
\makeatother


\begin{document}
\title{The obstacle problem and the Perron Method for nonlinear fractional equations in the Heisenberg group\thanks{The author is member of~\,``Gruppo Nazionale per l'Analisi Matematica, la Probabilit\`a e le loro Applicazioni (GNAMPA)'' of Istituto Nazionale di Alta Matematica (INdAM), with the support of the project ``Fenomeni non locali in problemi locali'', CUP\!\char`_E55F22000270001\,.}}

    \author{Mirco Piccinini}

\institute{M. Piccinini \at 
         Dipartimento di Scienze Matematiche, Fisiche e Informatiche,\\
          Universit\`a di Parma\\
          Parco Area delle Scienze 53/a, Campus, 43124 Parma, Italy\\
          \email{mirco.piccinini@unipr.it}
          }

\titlerunning{Nonlinear fractional equations in~$\h^n$}
\maketitle
\begin{abstract}
	We study the obstacle problem related to a wide class of nonlinear integro-differential operators, whose model is the fractional subLaplacian in the Heisenberg group. We prove both the existence and uniqueness of the solution, and that solutions inherit regularity properties of the obstacle such as boundedness, continuity and H\"older continuity up to the boundary. We also prove some independent properties of weak supersolutions to the class of problems we are dealing with. Armed with the aforementioned results, we finally investigate the Perron-Wiener-Brelot generalized solution by proving its existence for very general boundary data.

    \keywords{Nonlinear nonlocal operators \and Heisenberg group \and obstacle problem \and Perron method \and fractional subLaplacian.}

   \subclass{35A01 \and 35B45 \and  35R03 \and 47G20 \and 35R11 \and 31D05}

\end{abstract}
	
	\setcounter{tocdepth}{2}
	\tableofcontents

\setcounter{equation}{0}\setcounter{theorem}{0}

   \section{Introduction}
   In the present paper we investigate a class of obstacle problems where the leading operator~$\l$ is given by
   \begin{equation}
   	\label{operatore}
   	\l u (\xi) := P.~\!V. \int_{\h^n}\frac{|u(\xi)-u(\eta)|^{p-2}(u(\xi)-u(\eta))}{\dd(\eta^{-1}\circ \xi)^{Q+sp}}\,{\rm d}\eta, \qquad \xi \in \h^n,
   \end{equation}
    with~$\dd$ being a homogeneous norm on the Heisenberg group~$\h^n$, in accordance with forthcoming~Definition~\ref{def_homnorm}, and~$Q$ being the homogeneous dimension of~$\h^n$. 
    \vspace{1mm}
   
    In order to state our main results about the obstacle problem related to~$\l$ some further notation is needed. Let~$\Omega \Subset \Omega'$ be bounded open sets,~$h : \h^n \rightarrow [-\infty,\infty)$ be an extended real-valued function, called the {\it obstacle}, and~$g \in W^{s,p}(\h^n)$. Define the class 
    \begin{equation}\label{obst class}
   	\mathcal{K}_{g,h}(\Omega,\Omega') := \big\{ u \in W^{s,p}(\Omega'): u \geq h \mbox{ a.~\!e. in } \Omega, \, u=g \mbox{ a.~\!e. in } \h^n \smallsetminus \Omega\big\}.
    \end{equation}
    and consider the functional~$\mathbf{A}: \mathcal{K}_{g,h}(\Omega,\Omega') \rightarrow [W^{s,p}(\Omega')]'$ given by
    \begin{eqnarray}\label{A}
    \langle\mathbf{A}(u),v\rangle &:=&  \int_{\h^n}\int_{\h^n}\frac{|u(\xi)-u(\eta)|^{p-2}(u(\xi)-u(\eta)) (v(\xi)-v(\eta))}{\dd(\eta^{-1}\circ \xi)^{Q+sp}} \dxieta\\*
    && - \int_{\Omega} f(\xi,u)v(\xi) \dxi. \nonumber
    \end{eqnarray}
     where the function~$f:\h^n \times \r \rightarrow \r$ satisfies
    \begin{align} 
    	&   f (\cdot,x)\ \mbox{is measurable  for all~$x \in \r$ and} \ f(\xi,\cdot) \ \mbox{is continuous for a.~\!e.~$\xi \in \h^n$;} \label{reg f} \\
    	&	f \equiv f(\cdot,x) \in L^\infty_{\rm loc}(\h^n), \qquad \mbox{for any}~x\in \r, \ \mbox{uniformly in~$\Omega$;}\label{bound f}\\
    	&	\left(f(\xi,x_1)-f(\xi,x_2)\right)(x_1-x_2) \leq 0, \qquad \forall x_1,x_2 \in \r \ \mbox{and for a.~\!e.~}\xi \in \h^n. \label{monotonicity}
    \end{align}
   	We say that~$u \in \mathcal{K}_{g,h}(\Omega,\Omega')$ is a {\it solution to the obstacle problem in~$\mathcal{K}_{g,h}(\Omega,\Omega')$} if
   	$$
   	\langle\mathbf{A}(u),v-u \rangle \geq 0, \qquad \forall v \in \mathcal{K}_{g,h}(\Omega,\Omega').
   	$$
   Our first main result concerns existence and uniqueness.
   \begin{theorem}[{\bf Existence and uniqueness}]\label{t1 obst existence}
   	Under conditions~\eqref{reg f},~\eqref{bound f} and~\eqref{monotonicity}, if the class~$\mathcal{K}_{g,h}(\Omega,\Omega')$ is not empty, then there exists a unique solution to the obstacle problem in~$\mathcal{K}_{g,h}(\Omega,\Omega')$.
    \end{theorem}
    Moreover, if one considers the Dirichlet problem related to~$\l$ in~\eqref{operatore} in~$\Omega$; i.~\!e.,
      \begin{equation}\label{problema2}
    	   	\begin{cases}
    		   		\l u   = f & \text{in} \ \Omega,\\[0.5ex]
    		   		u   =g & \text{in} \ \h^n \smallsetminus \Omega,
    		   	\end{cases}
    	   \end{equation}
    then, as immediate consequence of the previous theorem, we have
    \begin{corol}\label{sec3.1_corol1}
    	Let~$u$ be the solution to the obstacle problem in~$\mathcal{K}_{g,h}(\Omega,\Omega')$. Then~$u$ is weak supersolution to~\eqref{problema2} in~$\Omega$, according to Definition~{\rm \ref{solution to inhomo pbm}}.
    	
    	Moreover, if~$B_r \equiv B_r(\xi_0)\subset \Omega$ is such that
    	$$
    	\essinf_{B_r}(u-h) >0,
    	$$
    	then~$u$ is a weak solution to~\eqref{problema2} in~$B_r$. In particular, if~$u$ is lower semicontinuous and~$h$ is upper semicontinuous in~$\Omega$, then~$u$ is a weak solution to~\eqref{problema2} in~$\{u >h\} \cap \Omega$.
    \end{corol}

    Our second main result concerns the regularity (up to the boundary) of solutions to the obstacle problem in~$\mathcal{K}_{g,h}(\Omega,\Omega')$.~In the same spirit of regularity results up to the boundary for nonlocal equation, see~\cite{KKP16,KKP17}, a natural measure theoretic condition on the boundary of~$\Omega$ has to be assumed; that is, there exists~$\delta_\Omega \in (0,1)$ and~$r_0>0$ such that for every~$\xi_0 \in \partial \Omega$ 
    \begin{equation}\label{condizione misura}
   	\inf_{r \in (0,r_0)} \frac{|(\h^n \smallsetminus \Omega)\cap B_r(\xi_0)|}{|B_r(\xi_0)|} \geq \delta_\Omega.
    \end{equation}
    We have
   \begin{theorem}[{\bf Regularity up to the boundary}]\label{boundary reg}
   	Let~$s \in (0,1)$, $p \in (1,\infty)$ and, under assumptions~~\eqref{reg f}, \eqref{bound f} and~\eqref{monotonicity}, let~$u$ solves the obstacle problem in~$\mathcal{K}_{g,h}(\Omega,\Omega')$ with~$g \in \mathcal{K}_{g,h}(\Omega,\Omega')$ and $\Omega$ satisfying condition~\eqref{condizione misura}. If~$g$ is locally H\"older continuous (resp. continuous) in~$\Omega'$ and~$h$ is locally H\"older continuous (rips. continuous) in~$\Omega$ or~$h \equiv -\infty$, then~$u$ is locally H\"older continuous (resp. continuous) in~$\Omega'$.
   \end{theorem}
   
   Let us note that if we assume the obstacle~$h \equiv -\infty$, from Corollary~\ref{sec3.1_corol1}, we can see the theorem above as a boundary regularity result for weak solution to~\eqref{problema2}, which completes the properties obtained in~\cite{MPPP21}.
   
   \vspace{2mm}
   Theorems~\ref{t1 obst existence} and~\ref{boundary reg} extend to the non-homogeneous and non-Euclidean setting the results obtained by Korvenp\"a\"a, Kuusi  and Palatucci in~\cite{KKP16} for the obstacle problem related to a class of integro-differential operators whose model is the fractional $p$-Laplacian. As well as they extend to the fractional setting the regularity properties obtained for the obstacle problem in Carnot groups; see for instance~\cite{DGP07,DGS03,PV13} and the references therein. Moreover, as shown in~\cite{KKP16}, solutions to this class of problems inherit up to the boundary the regularity of the obstacle both in the case of local boundedness, continuity and H\"older continuity. It is however worth mentioning that our functional and involved geometrical settings are different than the ones considered in the aforementioned papers. Indeed, since we consider the non-homogeneous case and the function~$f \equiv f(\cdot,u)$ depends on the solution~$u$ itself, we have to take into account the monotonicity assumption in~\eqref{monotonicity} in order to extend the theory presented in~\cite{KKP16}. In particular,~\eqref{monotonicity} is needed in the proof of Theorem~\ref{t1 obst existence} to show that the operator~$\mathbf{A}$, defined in~\eqref{A}, is monotone (e.~\!g.~satisfies condition~\eqref{A monotone}).
  As for conditions~\eqref{reg f} and~\eqref{bound f}, they are crucial in showing that some regularity properties of weak solutions to~\eqref{problema2} such as boundedness, H\"older continuity and Harnack inequality hold true; see the recent papers~\cite{MPPP21,PP21}. In~\cite{PP21} it is proven a generalization of the nonlocal Harnack inequality obtained by Ferrari and Franchi in~\cite{FF15} for the linear case when~$p=2$ via the celebrated Caffarelli-Silvestre extension (see~\cite{CS07}), which in the more general nonlinear framework (when~$p\neq 2$) is not applicable.
   
   \vspace{2mm} 
   Armed with the results in Theorem~\ref{t1 obst existence} and~\ref{boundary reg} we are able to extend the classical Perron Method for the Dirichlet problem~\eqref{problema2}. As well known, this method is a classical procedure in Potential Theory which provides the existence of a solution for the Dirichlet problem related to the Laplace equation in an arbitrary open set~$\Omega \subset \r^n$ and for any boundary datum~$g$, without "regularity assumption". It simply works by finding the least superharmonic function greater or equal to the datum~$g$ on the boundary~$\partial \Omega$ and it can be applied once knowing some helpful properties of superharmonic functions, such as comparison and maximum principles, and some barriers techniques. Some generalization of the Perron method are available in the literature when one considers instead of the classical Laplace operator its nonlinear counterpart. In this sense, we refer the reader to the paper~\cite{GLM86} by Granlund, Lindqvist and Martio and to the monograph~\cite{HKM06} by Heinonen, Kilple\"ainen and Martio and the references therein, for the study of more general nonlinear operators, whose prototype is the $p$-Laplacian.
   
   \vspace{2mm}
   Recently, more and more attention has been focused on the study of nonlocal operators and their related fractional Sobolev spaces. For a deep analysis of the Dirichlet problem related to nonlocal and nonlinear operators whose model is the fractional $p$-Laplacian we recall the papers~\cite{LL16} by Lindgren and Lindqvist,~\cite{KKP17} by Korvenp\"a\"a, Kuusi and Palatucci (where they prove the existence of the generalized solution in the sense of Perron) and the recent one~\cite{Mou17} by Mou. For a brief introduction to fractional Sobolev spaces in the Heisenberg group we refer the interested reader to the works of Adimurthi and Mallick~\cite{AM18}, where the classical Sobolev and Morrey embeddings are proven for the fractional functional setting, the paper~\cite{WD20} for some properties of solutions to the nonlinear fractional Laplacian on~$\h^n$, and the papers~\cite{KS18,KS21} by Kassymov and Surgan. In particular, in~\cite{KS21} the authors prove an existence result for the Dirichlet problem using the Mountain Pass theorem and the Folland-Stein embedding theorem as main tools. Moreover,it is worth mentioning~\cite{FMPPS18}, where the authors proved, via semigroup theory, an asymptotic behaviour of fractional subLaplacians on Carnot groups when the differentiability exponent~$s \nearrow 1$. An analogous result has been obtained in~\cite{PP21} via Taylor polynomials for the particular case of the Heisenberg group and in the same paper has been used to prove a stability property of the Harnack inequality when~$s$ goes to~$1$.
   
   We aim to extend the existence results for the generalized Perron solution proved in~\cite{LL16,KKP17} for nonlocal and nonlinear operators in the Heisenberg group.
  
   \begin{theorem}[{\bf Resolutivity}]\label{perron}
   	Let~$s \in (0,1)$ and~$p \in (1,\infty)$. Then, under assumptions~~\eqref{reg f}, \eqref{bound f} and~\eqref{monotonicity}, it holds that $\overline{H}_g = \underline{H}_g=H_g$, with~$\overline{H}_g,\underline{H}_g$ defined in Definition~{\rm \ref{perron classes and solutions}}.
   	
   	Moreover,~$H_g$ is a continuous weak solution in~$\Omega$ to problem~\eqref{problema2}, according to Definition~{\rm \ref{solution to inhomo pbm}}.
   \end{theorem}
   
    A few further comments are in order. We remark that for the proof of Theorem~\ref{perron} above an approach similar to the one used in~\cite{KKP17}, where the class of~$(s,p)$-superharmonic functions is defined, is not possible. Indeed, due to the presence of the function~$f \equiv f(\cdot,u)$ in the Dirichlet problem~\eqref{problema2}, most of the properties of $(s,p)$-superharmonic functions, holding in the classical Euclidean and homogeneous case, are not easily extendable to our context. This has led us to generalized the strategy used in~\cite{LL16} where the author studied the non-homogeneous problem when~$f=f(x)$. The dependence of the function~$f$ of the solution~$u$ is a novelty with respect to the case considered in~\cite{LL16} and this feature changes drastically the background of the problem we are dealing with. Indeed, we are forced to assume~\eqref{monotonicity} in order to prove some classical basic technique in Potential Theory, such as comparison principle (see forthcoming Proposition~\ref{comparison a.e.}) and to apply the regularity results for weak solutions up to the boundary in Theorem~\ref{boundary reg}. We remind that if~$f=f(\xi)$ condition~\eqref{monotonicity} is trivially satisfied and we plainly obtain the same results as the ones  presented in~\cite{LL16}. Hence, our method is consistent.

    \vspace{2mm}
    Moreover, with substantially no modifications on the forthcoming proofs, in the same spirit of the paper~\cite{KMS15} we can replace~$\l$ in~\eqref{operatore} with the more general class of operators
    $$
    \l_ \varPhi u (\xi) = P.~\!V. \int_{\h^n} \varPhi(u(\xi)-u(\eta))K(\xi,\eta)\deta, \qquad \xi \in \h^n,
    $$
    where the kernel~$K : \h^n \times \h^n \rightarrow [0,+\infty)$ is a measurable function  satisfying 
   $$
    \Lambda^{-1} \leq K(\xi,\eta)|\eta^{-1} \circ \xi|_{\h^n}^{-Q-sp} \leq \Lambda, \qquad \mbox{for a.~\!e. } \xi,\eta \in \h^n,
    $$
    for some~$s \in (0,1)$, $p>1$, $\Lambda \geq 1$ and~$|\cdot|_{\h^n}$ defined in~{\rm\eqref{work norm}} and where the real function~$ \varPhi$  is continuous,~$ \varPhi(0)=0$, and it satisfies 
    $$
    \lambda^{-1} |t|^{p} \leq \varPhi(t)t \leq \lambda |t|^p \qquad \mbox{for every} \ t \in \r \smallsetminus \{0\},
    $$
    for some constant~$\lambda >1$.
    
   Note that all the previous conditions are verified when~$K$ does coincide with a homogeneous norm~$\dd$ on~$\h^n$ {\rm (}see Proposition~{\rm\ref{prop1}}{\rm )} and~$ \varPhi(t)=|t|^{p-2}t$. 
  
   \vspace{2mm} 
   {\it Open problems and further developments}.
   The analysis of regularity estimates for double phase equations and their fractional counterpart, as well as mean value properties for solutions to nonlinear fractional operators, and their stability in the linear case when~$p=2$, are still very studied issues; we refer the interested reader to~\cite{DM20,DFP19,BDV20,BS21,SM20} and the references therein. However, to our knowledge, there is not much literature about this problems in the non-Euclidean setting of Carnot groups. Then, we trust that the present paper, together with the aforementioned ones~\cite{MPPP21,PP21}, could be a starting point in investigating nonlocal and nonlinear (or fractional double phase) operators with more complex non-Euclidean underlying geometry and their mean value formulas.   Moreover, we hope that Theorem~\ref{perron} could be an incentive in developing an axiomatic Potential Theory for more general nonlinear and nonlocal operators on Carnot group, as the well known results for subelliptic Laplacians presented in the monograph~\cite{BLU07}.
   
    \vspace{2mm}
    {\it The article is organized as follows}. In Section~\ref{sec_preliminaries} we introduce some basic notion about the Heisenberg group and the functional setting of the problem we are dealing with. We also recall some recent results proved for weak solutions to problem~\eqref{problema2}. In Section~\ref{sec3.1} we prove Theorem~\ref{t1 obst existence} and Theorem~\ref{boundary reg}. In Section~\ref{sec_weak_supersol} we prove some properties about weak supersolutions, and in Section~\ref{sec_perron} we give the proof of Theorem~\ref{perron}.
       
    \section{Preliminaries}\label{sec_preliminaries}
    We begin fixing the notation used throughout the paper. We denote with~$\c$ a general positive constant which can change form line to line. For the sake of readability, dependencies of the constants will be often omitted within the chains of estimates, therefore stated after the estimate. Moreover, throughout the following we indicate with~$u_-:= \max(-u,0)$ the negative part of~$u$ and with~$u_+:=\max(u,0)$ its positive part.
    \subsection{The setting of the main problem}
    Let us introduce the underlying geometrical and functional setting of our problem. For a more detailed presentation of the results cited below about the Heisenberg group we refer the reader to the monograph~\cite{BLU07}.
    
    We denote points in~$\r^{2n+1}$ as
    $$
     \xi := (z,t) = (x_1,\dots,x_n,y_1,\dots,y_n,t).
    $$
   The Heisenberg group~$\h^n$ is defined as the triple $(\r^{2n+1}, \circ, \left\{\Phi_\lambda\right\}_{\lambda >0})$, where the related group law is given by
    $$
	\xi \circ \xi' :=\big(x+x',\, y+y',\, t+t'+2\langle y,x'\rangle-2\langle x,y'\rangle \big)
    $$
   whereas the dilation group~$\left\{\Phi_\lambda\right\}_{\lambda >0}$ is defined as follow
   $$
   \begin{aligned}
	{\Phi}_\lambda :  \r^{2n+1} &\longrightarrow \r^{2n+1}\\
	\xi &\longmapsto {\Phi}_\lambda(\xi):=\big(\lambda z, \,\lambda^2 t \big).
   \end{aligned}
   $$
   We indicate with~$Q:=2n+2$ the homogeneous dimension related to~$\{\Phi_\lambda\}_{\lambda >0}$.
   It can be checked that~$\h^n$ is a Carnot group with the following stratification of its Lie algebra~$\mathfrak{h}^n$
   $$
   \mathfrak{h}^n = \textup{span}\{X_1,\dots,X_{2n}\}\oplus \textup{span}\{T\},
   $$
   where
   $$
   X_j := \partial_{x_j} +2y_j \partial_t, \quad
   X_{n+j}:= \partial_{y_j}-2x_j \partial_t, \quad 1 \leq j \leq n, \quad
   T= \partial_t.
   $$
   Given a domain $\Omega \subset \h^n$, for $u\in C^1(\Omega;\,\r)$ we define the subgradient $\nabla_{\h^n} u$ by
   $$
    \nabla_{\h^n} u (\xi):= \big(X_1u(\xi),\dots, X_{2n}u(\xi)\big),
    $$
   and
   $$
   |\nabla_{\h^n}u|^2 := \sum_{j=1}^{2n}|X_ju|^2.
  $$
   \begin{defn}
	\label{def_homnorm}
	A \textup{homogeneous norm} on $\h^n$ is a continuous function {\rm (}with respect to the  Euclidean topology\,{\rm )} $\dd : \h^n \rightarrow [0,+\infty)$ such that:
	\begin{enumerate}[\rm(i)]
		\item{
			$\dd({\Phi}_\lambda(\xi))=\lambda \dd(\xi)$, for every $\lambda>0$ and every~$\xi \in \h^n$;
		}
		\item{
			$\dd(\xi)=0$ if and only if $\xi=0$.
		}
	\end{enumerate}
	Moreover, we say that the homogeneous norm~$\dd$ is {\rm symmetric} if
	$$
   \dd(\xi^{-1})=\dd(\xi), \qquad \forall\xi \in \h^n.
	$$
   \end{defn}
	Let $\dd$ be a homogeneous norm on $\h^n$. Then the function
	$$
	\texttt{d} : \h^n \times \h^n \longmapsto [0,+\infty), \qquad \texttt{d}(\xi,\eta):=\dd (\eta^{-1}\circ \xi),
	$$
	is a pseudometric on $\h^n$. We will consider throughout the following the standard homogeneous norm on~$\h^n$, 
    \begin{equation}
	\label{work norm}
	|\xi|_{\h^n}= \left(|z|^4 +t^2\right)^\frac{1}{4}, \qquad \mbox{for any}~\xi=(z,t) \in \h^n.
    \end{equation}
    Fixed $\xi_0 \in \h^n$ and $R>0$, the ball~$B_R(\xi_0)$ with center $\xi_0$ and radius $R$ is given by
    $$
    B_R(\xi_0):=\Big\{\xi \in \h^n : |\xi_0^{-1}\circ \xi|_{\h^n} < R\Big\}.
    $$
    We conclude this section with some properties of the homogeneous norm on $\h^n$ that will be useful in the rest of the paper. For a proof of the next proposition we refer to~\cite[Proposition~5.1.4]{BLU07}.
    \begin{prop}
	\label{prop1}
	Let $\dd$ be a homogeneous norm on $\h^n$. Then there exists a constant $\Lambda>0$ such that
	\begin{equation}\label{def_lambda}
	\Lambda^{-1} \, |\xi|_{\h^n}\leq \dd(\xi) \leq \Lambda \, |\xi|_{\h^n}, \qquad \forall \xi \in \h^n,
	\end{equation}
	where~$|\cdot|_{\h^n}$ is defined in~\eqref{work norm}.
    \end{prop}

    \begin{rem}
    In view of the previous proposition, in most of the forthcoming proofs one can simply take into account the pure homogeneous norm defined in~\eqref{work norm} with no modifications at all. Moreover, the homogeneous norm~\eqref{work norm} turns out to be actually a norm, since it satisfies the triangle inequality as proved in~{\rm \cite{BFS17}}.
    \end{rem}
    We conclude this section with two results which turn out very useful in estimating the long and small-range contributions given by the homogeneous norm~$|\cdot|_{\h^n}$. 
   \begin{lemma}\label{lem1}
	Let $\gamma >0$ and let~$|\cdot|_{\h^n}$ be the homogeneous norm on $\h^n$ defined in{\rm~\eqref{work norm}}. Then, there exists~$\c =\c(n,\gamma)>0$ such that
	\begin{equation*}
		\int_{\h^n \smallsetminus B_r(\xi_0)} \frac{{\rm d}\xi}{|\xi_0^{-1}\circ \xi|_{\h^n}^{Q+\gamma}} \leq \c\, r^{-\gamma}.
	\end{equation*}
    \end{lemma}
   
    For the proof of the previous lemma we refer to~\cite[Lemma~2.6]{MPPP21}.
    
   \vspace{2mm}
   Let us introduce now our fractional functional setting. We remind that we refer to~\cite{AM18,KS18} for a more detailed presentation of the theorems and definitions stated below.
   
   Let~$p \in (1, \infty)$ and~$s \in (0,1)$ and define the fractional Sobolev spaces~$W^{s,p}(\h^n)$ on the Heisenberg group as the space
   \begin{equation}
	W^{s,p}(\h^n):=\left\{u \in L^p(\h^n):  \frac{|u(\xi)-u(\eta)|}{|\eta^{-1}\circ \xi|_{\h^n}^{\frac{Q}{p}+s}} \in L^p(\h^n \times \h^n)\right\},
   \end{equation}
   endowed with the natural fractional norm
   \begin{align}
	\|u\|_{W^{s,p}(\h^n)} & := \Big(\|u\|_{L^p(\h^n)}^p+[u]_{W^{s,p}(\h^n)}^p\Big)^\frac{1}{p}\notag\\
	& = \left( \int_{\h^n} |u(\xi)|^p \dxi + \int_{\h^n}\int_{\h^n}\frac{|u(\xi)-u(\eta)|^p}{|\eta^{-1}\circ \xi|_{\h^n}^{Q+sp}}\dxieta\right)^\frac{1}{p}, \qquad u \in W^{s,p}(\h^n).
   \end{align}
   In a similar way, given a domain $\Omega \subset \h^n$, one can define the  fractional Sobolev space $W^{s,p}(\Omega)$ in the natural way, as follows
   \begin{equation}
	W^{s,p}(\Omega):=\left\{u \in L^p(\Omega): \,\frac{|u(\xi)-u(\eta)|}{|\eta^{-1}\circ \xi|_{\h^n}^{\frac{Q}{p}+s}} \in L^p(\Omega \times \Omega)\right\},
   \end{equation}
   with norm given by
	 \begin{align}
		\|u\|_{W^{s,p}(\Omega)} & := \left(\|u\|_{L^p(\Omega)}^p+[u]_{W^{s,p}(\Omega)}^p\right)^\frac{1}{p}\notag\\
		& = \left( \int_{\Omega} |u(\xi)|^p \dxi + \int_{\Omega}\int_{\Omega}\frac{|u(\xi)-u(\eta)|^p}{|\eta^{-1}\circ \xi|_{\h^n}^{Q+sp}}\dxieta\right)^\frac{1}{p}, \qquad u \in W^{s,p}(\Omega).
	\end{align}
    We denote with $W^{s,p}_0(\Omega)$ the closure of $C_0^\infty(\Omega)$ in $W^{s,p}(\h^n)$ and recall that if $v \in W^{s,p}(\Omega')$ with  $\Omega \Subset \Omega'$ and $v=0$ outside of $\Omega$ almost everywhere, then $v$ has a representative in $W_0^{s,p}(\Omega)$ as well.
    As in the Euclidean case, a Sobolev inequality and a Poincaré-type one do still hold in the fractional setting in the Heisenberg group. Indeed -- as seen in~\cite[Theorem~2.5]{KS18} by extending the approach in~\cite[Theorem~6.5]{DPV12}; see also the Appendix in~\cite{PSV13} -- one can prove the estimate stated below.
   \begin{theorem}
	\label{sobolev}
	Let $p>1$ and $s \in (0,1)$ such that $sp<Q$. For any measurable compactly supported function $u : \h^n \rightarrow \r$ there exists~$\c=\c(n,p,s)>0$ such that
	\begin{equation*}
		\|u\|^p_{L^{p^*}(\h^n)}\, \leq \, \c \,[u]^p_{W^{s,p}(\h^n)}\,,
	\end{equation*}
	where $p^* = {Qp}/{(Q-sp)}$ is the critical Sobolev exponent.
   \end{theorem}
   Moreover, applying the same technique used in~\cite{Min03} we can prove a fractional Poincaré-type inequality 
   \begin{prop}\label{poincare}
    Let~$p \geq 1$ and~$s \in (0,1)$ and~$u \in W^{s,p}_{\rm loc}(\Omega)$. Then, for any $B_r \equiv B_r(\xi_0) \Subset \Omega$ we have that
   \begin{equation}
   \int_{B_r}|u-(u)_r|^p \dxi \leq \c \, r^{sp} \int_{B_r}\int_{B_r} \frac{|u(\xi)-u(\eta)|^p}{|\eta^{-1}\circ \xi|_{\h^n}^{Q+sp}} \dxieta,
   \end{equation}
   where~$\c=\c(n,p)>0$ and~$(u)_r = \dashint_{B_r}u \dxi$.
   \end{prop}
   \vspace{2mm}
 
    For any $v \in L^{p-1}_{\rm loc}(\h^n)$ and for any $B_r(\xi_0) \subset \h^n$ define the \textit{nonlocal tail of a function v in the ball}  $B_r(\xi_0)$ the quantity
     \begin{equation}
     	\label{tail}
     	\textup{Tail}(u;\xi_0,r):= \left(r^{sp} \int_{\h^n \smallsetminus B_r(\xi_0)}|v(\xi)|^{p-1}|\xi_0^{-1} \circ \xi|_{\h^n}^{-Q -sp}\,{\rm d}\xi\right)^{\frac{1}{p-1}}.
     \end{equation}
    This quantity is the analogue in the Heisenberg framework of the one introduced in~\cite{DKP16,DKP14}; see also~\cite{MPPP21,PP21}. The tail function has been crucial in studying the long-range interactions that naturally arise when dealing with operator as in~\eqref{operatore}; we refer to~\cite{Pal18} for a survey in such a framework. Let us consider the space of functions with finite tail
    $$
      L^{p-1}_{sp}(\h^n):=\left\{ v\in L^{p-1}_{\rm loc}(\h^n): \int_{\h^n} \frac{|v(\xi)|^{p-1}}{(1+|\xi|_{\h^n})^{Q+sp}}\dxi < \infty\right\}.
    $$
    Clearly, it follows by definition that~$L^\infty(\h^n) \subset L^{p-1}_{sp}(\h^n)$; so the  space above is not empty.
    
    With this bit of notation we can eventually give the definition of  weak (super/sub) solution to problem~\eqref{problema2}.
    
	\begin{defn}\label{solution to inhomo pbm}
		A function $u \in W^{s,p}_{\rm loc}(\Omega)$ such that~$u_- \in L^{p-1}_{sp}(\h^n)$ ( $u_+ \in L^{p-1}_{sp}(\h^n)$, resp.) is a {\rm fractional weak $p$-supersolution} (\textup{$p$-subsolution}, resp.) to~\textup{\eqref{problema2}} if
		\begin{eqnarray*}
		&& \int_{\h^n}\int_{\h^n}\frac{\big|u(\xi)-u(\eta)\big|^{p-2}\big(u(\xi)-u(\eta)\big)\big(\psi(\xi)-\psi(\eta)\big)}{\dd(\eta^{-1}\circ \xi)^{Q+sp}}  \,{\rm d}\xi \,{\rm d}\eta \\*[0.5ex]
		&& \hspace{6cm} \geq \big(\leq,\textrm{resp.}\big) \int_{\h^n}f(\xi,u(\xi))\psi(\xi) \, \,{\rm d}\xi,
	\end{eqnarray*}
		for any nonnegative  $ \psi \in C^\infty_0(\Omega)$.
		\\	A function u is a {\rm fractional weak $p$-solution} if it is both a fractional weak $p$-super and $p$-subsolution.
	\end{defn}
    \vspace{2mm}
    
    We often simply write weak supersolution, subsolution and solution, omitting~$p$ from our notation. Moreover, in the definition above we can replace the condition~$\psi \in C^\infty_0(\Omega)$ with~$\psi \in W^{s,p}_0(D)$ with~$D \Subset \Omega$.
  
    As the next lemma shows, it does not make any difference to assume that~$u \in L^{p-1}_{sp}(\h^n)$ instead of~$u_-$ in Definition~\ref{solution to inhomo pbm}.
    
    \begin{lemma}\label{s1 lem1}
    Let~$u$ be a weak supersolution in~$B_{2r} \equiv B_{2r}(\xi_0)$. Then, there exists~$\c=\c(n,p,s, \Lambda)>0$ such that
    \begin{align*}
    &\textup{Tail}(u_+;\xi_0,r)\\
           & \leq \c \, \left(r^\frac{sp-1-Q}{p-1}[u]_{W^{h,p-1}(B_r)}+r^{-\frac{Q}{p-1}}\|u\|_{L^{p-1}(B_r)} + \textup{Tail}(u_-;\xi_0,r)+ r^\frac{sp}{p-1} \, \|f\|_{L^\infty(B_r)}^\frac{1}{p-1}\right)
    \end{align*}
   with
   $$
   h = \max\left(0,\frac{sp-1}{p-1}\right) < s.
   $$
   In particular, if~$u$ is a weak supersolution in an open set~$\Omega$, then~$ u \in L^{p-1}_{sp}(\h^n)$.
    \end{lemma}
   \begin{proof}
   Let us consider in the weak formulation a function~$\psi \in C^\infty_0(B_{r/2})$ such that~$\psi \equiv 1$ in~$B_{r/4}$, with~$0 \leq \psi \leq 1$ and~$|\nabla_{\h^n} \psi | \leq 8/r$. We have
   \begin{align}\label{s1 lem1 e1}
   \int_{\h^n}f(\xi,u)\psi(\xi) \dxi & \leq \int_{B_r}\int_{B_r}\frac{|u(\xi)-u(\eta)|^{p-2}(u(\xi)-u(\eta))(\psi(\xi)-\psi(\eta))}{\dd(\eta^{-1}\circ \xi)^{Q+sp}}\dxieta\notag\\*
   &\quad +2 \int_{\h^n \smallsetminus B_r}\int_{B_{r/2}}\frac{|u(\xi)-u(\eta)|^{p-2}(u(\xi)-u(\eta))\psi(\xi)}{\dd(\eta^{-1}\circ \xi)^{Q+sp}}\dxieta\notag\\
   & =: I_1+I_2
   \end{align}
   We separately estimate~$I_1$ and~$I_2$ on the righthand side of~\eqref{s1 lem1 e1}. As for~$I_1$ using the fact that~$|\psi(\xi)-\psi(\eta)| \leq |\eta^{-1}\circ \xi|_{\h^n}/r$ we trivially obtain that
   $$
   I_1 \leq \frac{\c}{r^{\min(sp,1)}}[u]_{W^{h,p-1}(B_r)}^{p-1}.
   $$
   On the other hand, for the second integral~$I_2$ we make use of the inequality
   $$
   |u(\xi)-u(\eta)|^{p-2}(u(\xi)-u(\eta)) \leq 2^{p-1}( u_+^{p-1}(\xi)+u_-^{p-1}(\eta))-u_+^{p-1}(\eta),
   $$
   Hence, we obtain that
   $$
   I_2 \leq \c \,  r^{-sp}\|u\|^{p-1}_{L^{p-1}(B_r)} +\c \, r^{Q-sp} \textup{Tail}(u_-;\xi_0,r)^{p-1} -\c \, r^{Q-sp} \textup{Tail}(u_+;\xi_0,r)^{p-1}.
   $$
   As for the lefthand side, using~\eqref{bound f}, we eventually obtain that
   $$
   \int_{\h^n} f(\xi,u) \psi(\xi) \dxi \geq - \c \, r^Q\|f\|_{L^{\infty}(B_r)}.
   $$
   Combining all previous estimates we get the desired inequality. The second statement follows by applying H\"older's Inequality.
   \end{proof}
   
   \subsection{Algebraic inequalities}
   We state here some practical and elementary inequalities. For the sake of readability we adopt the following notation
   \begin{equation}\label{sec3.1_e0}
   L(a,b) := |a-b|^{p-2}(a-b), \qquad a,b \in \r.
   \end{equation}
   For a proof of the following lemma see~\cite[Lemma~2.1-2.2]{KKP16}
   \begin{lemma}\label{s1 lem2}
	Let~$1 < p \leq 2$ and~$a,b,a',b' \in \r$. Then,
	\begin{equation}
		|L(a,b)-L(a',b')| \leq 4|a-a'-b+b'|^{p-1}.
	\end{equation}
	Let~$p \geq 2$ and~$a,b,a',b' \in \r$. Then,
	\begin{equation}
		|L(a,b)-L(a',b')| \leq \c \,|a-a'|^{p-1} +\c \,|a-a'||a-b|^{p-2},
	\end{equation}
	and
	\begin{equation}
		|L(a,b)-L(a',b')| \leq \c \, |b-b'|^{p-1}+\c \,|b-b'||a-b|^{p-2},
	\end{equation}
	where~$\c$ depends only on~$p$.
   \end{lemma}
   
   Finally we would like to recall the following inequalities. The first states that there exists a positive constant~$\c$ depending only on~$p$ such that
   
   \begin{equation}\label{s1 e1}
   	\frac{1}{\c} \leq \frac{(|a|^{p-2}a-|b|^{p-2}b)(a-b)}{(|a|+|b|)^{p-2}(a-b)^2} \leq \c,
   \end{equation}
   when~$a,b \in \r$, $a \neq b$. In particular,
   
   \begin{equation}\label{s1 e2}
   	(|a|^{p-2}a-|b|^{p-2}b)(a-b) \geq 0, \qquad a,b \in \r.
   \end{equation}

   \subsection{Some results about nonlocal fractional equations in $\h^n$}
    In this section we state some recent results about nonlocal fractional equation in the Heisenberg group. All the following theorems are contained in the recent works~\cite{MPPP21,PP21}.
  
   We remark that the assumption~$u \in W^{s,p}(\h^n)$ can be weakened requiring instead~$u \in W^{s,p}_{\rm loc}(\Omega) \cap L^{p-1}_{sp}(\h^n)$ with substantially no modification in the proofs of the next theorems. Moreover, all results hold only requiring that the function~$f$ satisfies conditions~\eqref{reg f} and~\eqref{bound f}.
  
   We begin stating a Caccioppoli-type inequality with tail. 
   
   	\begin{theorem}[{\bf Caccioppoli-type inequality with tail}]\label{teo_caccioppoli}
   	Let  $s \in (0,1)$, $p \in (1,\infty)$, and let $u$ be a weak subsolution to~\eqref{problema2}. Then, for any~$B_r \equiv B_r(\xi_0) \subset \Omega$ and any~nonnegative $\p \in C^\infty_0(B_r)$, the following estimate holds true	
   	\begin{align}\label{caccioppoli}
   		\int_{B_r}& \int_{B_r} |\eta^{-1} \circ \xi|_{\h^n}^{-Q-sp}  |w_+(\xi)\p(\xi)-w_+(\eta)\p(\eta)|^p \, \,{\rm d}\xi \,{\rm d}\eta\notag\\*
   		& \leq \c\int_{B_r}\int_{B_r} |\eta^{-1} \circ \xi|_{\h^n}^{-Q-sp} w_+^p(\xi)|\p(\xi)-\p(\eta)|^p \, \,{\rm d}\xi \,{\rm d}\eta\\*
   		&\quad+ \c \int_{B_r}w_+(\xi)\p^p(\xi) \, \,{\rm d}\xi \biggl(\sup_{\eta \in \textup{supp}\, \p}\int_{\h^n \smallsetminus B_r} |\eta^{-1} \circ \xi|_{\h^n}^{-Q-sp} w_+^{p-1}(\xi) \, \,{\rm d}\xi\notag   \\*
   		& \hspace{4.5cm} +\|f\|_{L^\infty(B_r)}\biggr)\notag
   	\end{align}
   	where $w_+ := (u-k)_+$ and $\c=\c\,(n,p,s,\Lambda)>0$.
   \end{theorem}

   \vspace{2mm} 
   The main consequence of the previous inequality is a local boundedness result (which is new even for the linear case~$p =2$). In particular,
   
   \begin{theorem}[{\bf Local boundedness}]\label{teo_bdd} 
   	Let $s \in (0,1)$ and $p \in (1,\infty)$, let $u$ be a weak subsolution to~{\rm \eqref{problema2}} and let $B_r \equiv B_r(\xi_0)   \subset \Omega$. Then the following estimate holds true, for any $\delta \in (0,1]$,
   	\begin{equation}		\label{eq_bdd}
   		\esssup_{B_{r/2}}u \, \leq\,   \delta \,\textup{Tail}(u_+; \xi_0, r/2) + \c\,\delta^{-\frac{(p-1)Q}{sp^2}}  \left( \,\dashint_{B_r}u_+^p{\rm d}\xi\right )^\frac{1}{p}, 
   	\end{equation}
   	where $\textup{Tail}(u_+;\xi_0,r/2)$ is defined in~\eqref{tail} and the positive constant $\c$ depends only on $n,p,s,\Lambda$ and $\|f\|_{L^\infty(B_r)}$.
   \end{theorem}

   \vspace{2mm}
    The parameter~$\delta$ in the theorem above gives a precise interpolation between the local and nonlocal term. Combining together the theorems above with a {\it Logarithmic-Lemma with tail}~\cite[Lemma~1.4]{MPPP21}, we can prove that weak solutions to~\eqref{problema2} enjoy oscillation estimates, that clearly holds H\"older continuity, and a natural (nonlocal) Harnack inequality. 
    
    \begin{theorem}[{\bf H\"older continuity}]\label{teo_holder}
    	Let $s \in (0,1)$, $p \in (1,\infty)$, and let $u$ be a weak solution to~{\rm \eqref{problema2}}. Then $u$ is locally H\"older continuous in~$\Omega$. In particular, if $B_{2r} \equiv B_{2r}(\xi_0) \Subset \Omega$ then there are constants $\alpha < sp/(p-1)$ and $\c>0$, both depending only on $n,p,s,\Lambda$ and $\|f\|_{L^\infty(B_r)}$, such that 
    	\begin{equation}
    		\osc\limits_{B_\varrho} \, u 
    		\, \leq\,  \c\, \left(\frac{\varrho}{r}\right)^\alpha \left[\textup{Tail}(u;\xi_0,r) + \left(\,\dashint_{B_{2r}}|u|^p \,{\rm d}\xi\right)^\frac{1}{p}\right],
    	\end{equation} 
    	for every $\varrho\in (0,r)$, where  $\textup{Tail}(u_+;\xi_0,r/2)$ is defined in~\eqref{tail}.
    \end{theorem}

       \begin{theorem}[{\bfseries Nonlocal weak Harnack inequality}]\label{weak harnack}
   	For any $s \in (0,1)$ and any $p \in (1,\infty)$,	 let $u \in W^{s,p}(\h^n)$ be a weak supersolution to~\eqref{problema2} such that $u \geq 0$ in $B_R \subseteq \Omega$.
   	Then, for any $B_r$ such that $B_{6r}\subset B_R$, it holds
   	\begin{equation}\label{eq_weakharnack}
   		\left( \ \dashint_{B_{r}}  u^t\,{\rm d}\xi\,\right)^\frac{1}{t} \ \leq \ c\,\inf_{B_{\frac{3}{2}r}} u 
   		\, +c \left(\frac{r}{R}\right)^\frac{sp}{p-1} {\rm Tail}(u_-;\xi_0,R) \, + c\chi,
   	\end{equation}
   	where
   	$$
   	\chi = 
   	\begin{cases}
   		r^\frac{Qsp}{t(Q-sp)}\|f\|_{L^\infty(B_R)}^\frac{Q}{t(Q-sp)} \qquad \text{for} \ t<\frac{Q(p-1)}{Q-sp} &\quad  \text{if} \ sp<Q,\\[0.8ex]
   		r^\frac{Q(s-\epsilon)}{t\epsilon}\|f\|_{L^\infty(B_R)}^\frac{s}{t\epsilon}  \qquad \text{for any} \ s-{Q}/{p}< \epsilon<s \ \text{and}\  t<\frac{(p-1)s}{\epsilon} &\quad  \text{if} \ sp\geq Q,
   	\end{cases}
   	$$
   	${\rm Tail}(\cdot)$ is defined in~\eqref{tail}, and $u_-:=\max(-u,0)$ is the negative part of the function~$u$. The constant~$c$ depends only on $n$, $s$, $p$, and the structural constant~$\Lambda$ defined in~\eqref{def_lambda}.
   \end{theorem}

  \section{The obstacle problem}\label{sec3.1}
  We divide this section into two parts; proving first existence of solutions then their regularity (up to the boundary) properties.
  \subsection{Existence of solutions}
  Let us begin splitting the operator~$\mathbf{A}$ in~\eqref{A} as follows
   $$
   \langle\mathbf{A}_1(u),v\rangle := \int_{\Omega'}\int_{\Omega'}\frac{ L (u(\xi),u(\eta))(v(\xi)-v(\eta))}{\dd(\eta^{-1}\circ \xi)^{Q+sp}} \dxieta - \int_{\Omega}f(\xi,u) v(\xi) \dxi,
   $$
   and
   $$
   \langle\mathbf{A}_2(u),v\rangle := 2\int_{\h^n \smallsetminus \Omega'}\int_{\Omega} \frac{L (u(\xi),g(\eta))v(\xi)}{\dd(\eta^{-1}\circ \xi)^{Q+sp}} \dxieta,
   $$
   The following observation is useful in estimating the second term~$\mathbf{A}_2$.

   \begin{rem}\label{sec3.1_rem2}
   	Since~$\Omega \Subset \Omega'$ there exist~$\xi_0 \in \Omega$ and~$ r \in (0, r_0)$, with~$r_0 := \dist(\Omega,\partial \Omega')$, such that~$\Omega \Subset B_r(\xi_0) \Subset \Omega'$. Hence, for any~$\eta \in \h^n \smallsetminus B_r(\xi_0)$ and any~$\xi \in \Omega $, it holds
   	\begin{equation}\label{sec3.1_e3_norme}
   		\frac{|\eta^{-1}\circ \xi_0|_{\h^n}}{|\eta^{-1}\circ \xi|_{\h^n}} \leq 1+ \frac{|\xi^{-1}\circ \xi_0|_{\h^n}}{|\eta^{-1}\circ \xi_0|_{\h^n}-|\xi^{-1}\circ \xi_0|_{\h^n}} \leq  \c.
   	\end{equation}
   	Thus, by the inequality above, considering a function~$g \in L^{p-1}_{sp}(\h^n)$ and a function~$v \in L^p(\Omega)$, we have that
   	\begin{align*}
   		\int_{\h^n \smallsetminus \Omega'}\int_\Omega |g(\eta)|^{p-1}|v(\xi)| & |\eta^{-1}\circ \xi|_{\h^n}^{-Q-sp} \dxieta \\
   		& \leq \c\int_{\h^n \smallsetminus B_r(\xi_0)}\int_\Omega |g(\eta)|^{p-1}|v(\xi)| |\eta^{-1}\circ \xi_0|_{\h^n}^{-Q-sp} \dxieta\\
   		& \leq \int_{\h^n \smallsetminus B_r(\xi_0)}|g(\eta)|^{p-1}|\eta^{-1} \circ \xi_0|^{-Q-sp}\deta \int_{\Omega} |v(\xi)| \dxi\\
   		& \leq \c \, r^{-sp} {\rm Tail}(g;\xi_0,r)^{p-1}\|v\|_{L^p(\Omega)},
   	\end{align*}
   	by Jensen's Inequality, with~$\c=\c(p,\Omega)>0$.
   \end{rem} 

   \begin{proof}[Proof of Theorem~{\rm \ref{t1 obst existence}}]
    We begin proving that the operator~$\mathbf{A}$, defined in~\eqref{A}, is monotone, weakly continuous and coercive on the set~$\mathcal{K}_{g,h}(\Omega,\Omega')$.
    We first prove that the operator~$\mathbf{A}$ in~\eqref{A} is monotone, weakly continuous and coercive.
    
   First we prove that~$\mathbf{A}$ is monotone, i.~\!e. 
   \begin{equation}\label{A monotone}
   \langle \mathbf{A}(u) -\mathbf{A}(v), u-v\rangle \geq 0 \qquad \mbox{for every~$u,v \in \mathcal{K}_{g,h}(\Omega,\Omega')$.}
   \end{equation}
   	
   	Let us fix two functions~$u,v \in \mathcal{K}_{g,h}(\Omega,\Omega')$ and consider separately~$\mathbf{A}_1$ and~$\mathbf{A}_2$. For~$\mathbf{A}_1$ we have
   	\begin{align*}
   		\langle & \mathbf{A}_1(u) -\mathbf{A}_1(v), u-v\rangle \\*
     & = \int_{\Omega'} \int_{\Omega'} (L   (u(\xi),u(\eta)) - L(v(\xi),v(\eta)) )\notag\\*
   		& \hspace{1.5cm} \times(u(\xi) - u(\eta)-v(\xi)+v(\eta)) \dd(\eta^{-1} \circ \xi)^{-Q-sp}\dxieta\notag\\
   		& \quad -\int_{\Omega}(f(\xi,u)-f(\xi,v))(u(\xi)-v(\xi)) \dxi.
   	\end{align*}
   	By inequality~\eqref{s1 e2} we have that we only need to prove that
   	\begin{equation}\label{sec3.1_lem4_e1}
   		\int_{\Omega}(f(\xi,u)-f(\xi,v))(u(\xi)-v(\xi)) \dxi \leq 0.	
   	\end{equation}
   	But by condition~\eqref{monotonicity} there exists a negligible set~$E$ such that for any~$\xi \in \Omega\smallsetminus E$ it holds
   	$$
   	\big(f(\xi,u(\xi))-f(\xi,v(\xi))\big)\big(u(\xi)-v(\xi)\big) \leq 0.
   	$$
   	Hence~\eqref{sec3.1_lem4_e1} follows. Similarly, proceeding as in~\cite{KKP16}, it follows that~$\langle \mathbf{A}_2(u) -\mathbf{A}_2(v), u-v\rangle \geq 0$. Hence,~\eqref{A monotone} is proven.
   	
   	\vspace{2mm}
   	
   	Now, let us show that~$\mathbf{A}$ is weakly continuous. Let~$\{u_j\}$ be a sequence in~$\mathcal{K}_{g,h}(\Omega,\Omega')$ converging to~$u \in \mathcal{K}_{g,h}(\Omega,\Omega')$ in~$W^{s,p}(\Omega')$. We want to prove that
   	\begin{equation}\label{A weak cont}
   		 \langle \mathbf{A}(u_j ) -\mathbf{A}(u),v\rangle \rightarrow 0 \qquad \mbox{for all~$v \in W^{s,p}(\Omega')$.}
   	\end{equation}	 
     Let~$v \in W^{s,p}(\Omega')$. We estimate separately~$\mathbf{A}_1$ and~$\mathbf{A}_2$.
   	\begin{align*}
   		& |\langle \mathbf{A}_1(u_j)  -\mathbf{A}_1( u), v\rangle|\notag\\
   		&  \leq \int_{\Omega'}\int_{\Omega'} |L(u_j(\xi),u_j(\eta))-L(u(\xi),u(\eta))||v(\xi)-v(\eta)| \dd(\eta^{-1} \circ \xi)^{-Q-sp} \dxieta\notag\\
   		& \quad +  \int_{\Omega}|f(\xi,u_j)-f(\xi,u)||v(\xi)| \dxi \notag\\
   		& =: I_1 +I_2.
   	\end{align*}
   	Consider~$I_2$. Since~$u_j$ converges to~$u$ in~$W^{s,p}(\Omega')$, it also converges to~$u$ in~$L^p(\Omega)$.

   	Thus, there exists a subsequence~$\{u_{j_k}\}_k$ which converges to~$u$ almost everywhere in~$\Omega$. Moreover, since by condition~\eqref{reg f} $x \mapsto f(\xi,x)$ is continuous for a.\!~e.~$\xi \in \Omega$, we have that
   	$$
   	f(\xi,u_{j_k}) \rightarrow f(\xi,u) \qquad \mbox{for a.\!~e. } \xi \in \Omega.
   	$$  
   	Hence, since the sequence $\{f(\cdot,u_{j_k})\}$ has $L^{p/(p-1)}(\Omega)$-norms uniformly bounded thanks to~\eqref{bound f}, we have that $f(\xi,u_{j_k}) \rightarrow f(\xi,u)$ weakly in~$L^{p/(p-1)}(\Omega)$. Since the weak limit is independent of the choice of the subsequence it follows that~$f(\xi,u_{j}) \rightarrow f(\xi,u)$ weakly in~$L^{p/(p-1)}(\Omega)$. Then,
   	$$
   	I_2 \xrightarrow{j \rightarrow \infty} 0.
   	$$                  	                                                                   	Let us focus on~$I_1$. By Lemma~\ref{s1 lem2} and H\"older's Inequality we obtain, for~$1<p\leq 2$,
   	\begin{align*}
   		I_1 & \leq \c \int_{\Omega'}\int_{\Omega'}|u_j(\xi)-u(\xi)-(u_j(\eta)-u(\eta))|^{p-1}|v(\xi)-v(\eta)| \frac{\dxieta}{|\eta^{-1}\circ \xi|_{\h^n}^{Q+sp}}\notag\\
   		& \leq \c \, \|u_j-u\|_{W^{s,p}(\Omega')}^{p-1}\|v\|_{W^{s,p}(\Omega')} \xrightarrow{j \rightarrow \infty}0.
   	\end{align*}
   	When~$p \geq 2$ by~\eqref{s1 lem2}, proceeding as done in~\cite{KKP16}, we eventually arrive at
    $$
   		I_1 \leq \c \, (\|u_j\|_{W^{s,p}(\Omega')}+\|u\|_{W^{s,p}(\Omega)})^{p-2}\|u_j-u\|_{W^{s,p}(\Omega')}\|v\|_{W^{s,p}(\Omega')},
   	$$
   	which also vanishes when~$j \rightarrow \infty$. Also for~$\mathbf{A}_2$
   	we split the case~$1 < p \leq 2$ and~$p>2$. For~$1 < p \leq 2$ using again Lemma~\ref{s1 lem2}, we obtain
   	\begin{align*}
   		& |\langle \mathbf{A}_2(u_j)  -\mathbf{A}_2( u), v\rangle|\\
   		& \leq 2\int_{\h^n \smallsetminus \Omega'}\int_\Omega |L(u_j(\xi),g(\eta)) -L(u(\xi),g(\eta))||v(\xi)| \dd(\eta^{-1} \circ \xi)^{-Q-sp} \dxieta\\
   		& \leq \c \int_{\h^n \smallsetminus \Omega'}\int_\Omega |u_j(\xi)-u(\xi)|^{p-1}|v(\xi)|\frac{\dxieta}{|\eta^{-1}\circ \xi|_{\h^n}^{Q+sp}}\\
   		& \leq \c  \,\|u_j-u\|_{W^{s,p}(\Omega')}^{p-1}\|v\|_{W^{s,p}(\Omega')} \xrightarrow{j \rightarrow \infty} 0.
   	\end{align*}
   	When~$p \geq 2$, proceeding as done in~\cite{KKP16}, by Remark~\ref{sec3.1_rem2}, we have that
   	\begin{align*}
   		&|\langle \mathbf{A}_2(u_j) -\mathbf{A}_2 (u), v\rangle|\\*
     & \leq \c \, \|u_j-u\|_{W^{s,p}(\Omega')}^{p-1}\|v\|_{W^{s,p}(\Omega')}\\
   		& \quad +\c \,\|u_j-u\|_{W^{s,p}(\Omega')}\|u\|_{W^{s,p}(\Omega')}^{p-2}\|v\|_{W^{s,p}(\Omega')}\\
   		& \quad + \c \, r^{-sp}{\rm Tail}(g;\xi_0,r)^{p-2}\|u_j-u\|_{W^{s,p}(\Omega')}\|v\|_{W^{s,p}(\Omega')}\xrightarrow{j \rightarrow \infty} 0.
   	\end{align*}
   	Thus,~\eqref{A weak cont} holds.
   	
   	\vspace{2mm}
   	
   	In the end, let us prove the coercivity, i.~\!e. we prove that there exists a function~$v \in \mathcal{K}_{g,h}(\Omega,\Omega')$ such that
   	\begin{equation}\label{A coercive}
   	\frac{\langle \mathbf{A}(u)-\mathbf{A}(v),u-v \rangle}{\|u-v\|_{W^{s,p}(\Omega')}} \rightarrow \infty \quad {\rm as} \quad \|u\|_{W^{s,p}(\Omega')} \rightarrow \infty.
   	\end{equation}
   	Since we have assumed that the class~$\mathcal{K}_{g,h}(\Omega,\Omega')$ is not empty, there exists at least one~$v \in \mathcal{K}_{g,h}(\Omega,\Omega')$. Let us fix such~$v$. Following the same chain of estimates used in~\cite{KKP16} it can be proved that there exists a constant~$\c$ independent of~$u$ such that
   	
   	\begin{equation}\label{sec3.1_lem4_e5}
   		|\langle \mathbf{A}(u)-\mathbf{A}(v),v \rangle| \leq \c \, \|u\|_{W^{s,p}(\Omega')}^{p-1}+\c.
   	\end{equation}
   	We prove that the contribution from~$\langle \mathbf{A}(u)-\mathbf{A}(v),u \rangle$ dominates when~$\|u\|_{W^{s,p}(\Omega')}$ is large. As customary we split the operator~$\mathbf{A}$ into~$\mathbf{A}_1$ and~$\mathbf{A}_2$. Then,
   	\begin{align*}
   		&\langle \mathbf{A}_1(u)  -\mathbf{A}_1 (v), u\rangle\\
   		& =\int_{\Omega'}\int_{\Omega'} (L(u(\xi),u(\eta))-L(v(\xi),v(\eta)))(u(\xi)-u(\eta)) \dd(\eta^{-1} \circ \xi)^{-Q-sp} \dxieta\notag\\
   		& \quad - \int_{\Omega}(f(\xi,u)-f(\xi,v))u(\xi)\dxi\notag\\
   		& \geq \c \int_{\Omega'}\int_{\Omega'}\frac{|u(\xi)-u(\eta)|^p}{|\eta^{-1}\circ \xi|_{\h^n}^{Q+sp}}\dxieta -\c \int_{\Omega'}\int_{\Omega'}\frac{|v(\xi)-v(\eta)|^{p-1}|u(\xi)-u(\eta)}{|\eta^{-1}\circ \xi|_{\h^n}^{Q+sp}}\dxieta\\
   		& \quad- \|f\|_{L^\infty(\Omega)}\int_{\Omega'}|u(\xi)|\dxi\\
   		& \geq \c \, [u-g]_{W^{s,p}(\Omega')}^p -\c \ [g]_{W^{s,p}(\Omega')}^p -\c \ [v]_{W^{s,p}(\Omega')}^{p-1}[u]_{W^{s,p}(\Omega')} - \c  \ \|u\|_{L^p(\Omega')}^{p-1}\\
   		& \geq \c \ \|u\|_{W^{s,p}(\Omega')}^p -\c  \ \|g\|_{W^{s,p}(\Omega')}^p -\c \ \|v\|_{W^{s,p}(\Omega')}^{p-1}\|u\|_{W^{s,p}(\Omega')}
   	    - \c \ \|u\|_{W^{s,p}(\Omega')}^{p-1}
   	\end{align*}
   	where we have used H\"older's Inequality and the Sobolev embedding Theorem~\ref{sobolev} and where the constant~$\c$ depends on~$n,p,s,\|f\|_{L^\infty(\Omega)}$ and~$\Omega'$ and the structural constant~$\Lambda$. On the other hand, proceeding as in~\cite{KKP16} and recalling Remark~\ref{sec3.1_rem2}, for~$\mathbf{A}_2$, we have 
   	\begin{align*}
   		\langle \mathbf{A}_2( u )& -\mathbf{A}_2 (v),u \rangle \\
   		& \geq - \c  \ \|u\|_{L^p(\Omega')}^{p-1}\|v\|_{L^p(\Omega')} -\c  \ r^{-sp}{\rm Tail}(g,\xi_0,r)^{p-1}\|v\|_{L^p(\Omega')} -\|v\|^p_{L^p(\Omega')}.
   	\end{align*}
   	Combining all previous estimates we obtain that
   	$$
   	\langle \mathbf{A}(u) -\mathbf{A}(v),u-v \rangle \geq \c \ \|u\|^p_{W^{s,p}(\Omega')} -\c \ \|u\|^{p-1}_{W^{s,p}(\Omega')}-\c \ \|u\|_{W^{s,p}(\Omega)} -\c,
   	$$
   	where the constant~$\c$ is independent of~$u$. Hence, the proof of~\eqref{A coercive} is complete.
    
    \vspace{2mm}
    Then, since~$\mathbf{A}$ in~\eqref{A} is monotone, weakly continuous and coercive in~$\mathcal{K}_{g,h}(\Omega,\Omega')$, the existence and uniqueness of the solution~$u$ to the obstacle problem can be proved extending to our setting the same argument used in Theorem~1 in~\cite{KKP16}.
    \end{proof}

   We now prove Corollary~\ref{sec3.1_corol1}.
   \begin{proof}[Proof of Corollary~{\rm \ref{sec3.1_corol1}}]
   Let us prove that the solution~$u$ to the obstacle problem in~$\mathcal{K}_{g,h}(\Omega,\Omega')$ is a weak supersolution to~\eqref{problema2} according to Definition~\ref{solution to inhomo pbm}. First of all note that clearly~$u \in W^{s,p}_{\rm loc}(\Omega)  \cap L^{p-1}_{sp}(\h^n)$. Moreover, for any nonnegative~$\psi \in C^\infty_0(\Omega)$ the function  $v:= u+\psi \in \mathcal{K}_{g,h}(\Omega,\Omega')$. Hence,
   	\begin{align*}
   		0 & \leq \langle \mathbf{A}(u), v-u \rangle\\
   		& = \int_{\h^n}\int_{\h^n} L(u(\xi),u(\eta)) (\psi(\xi)-\psi(\eta)) \dd(\eta^{-1} \circ \xi)^{-Q-sp} \dxieta \\*
     &\quad -\int_{\h^n} f(\xi,u) \psi(\xi) 
   		\dxi.                   
   	\end{align*}
   	Then,~$u$ is a weak supersolution according to Definition~\ref{solution to inhomo pbm}. The last statement of Corollary~\ref{sec3.1_corol1} follows repeating the argument of Corollary~1 in~\cite{KKP16}.
    \end{proof}

   \subsection{Proof of Theorem~\ref{boundary reg}}
    We separately consider two cases: the case of the interior regularity and the case of the boundary regularity.
   \subsubsection{Interior regularity}   
   In order to prove that solutions to the obstacle problem are H\"older continuous or simply continuous inside~$\Omega$ we begin proving a local boundedness result.
   
   \begin{theorem}\label{sec4_thm1}
   	Let~$s \in (0,1)$,~$p \in (1,\infty)$ and, under hypothesis~\eqref{reg f},~\eqref{bound f} and~\eqref{monotonicity}, let~$u$ be the solution to the obstacle problem in~$\mathcal{K}_{g,h}(\Omega,\Omega')$. Assume that~$B_r \equiv B_r(\xi_0) \subset \Omega'$ and set
   	$$
   	M := \max \big(\esssup_{B_r \cap \Omega}h, \esssup_{B_r \smallsetminus \Omega} g\big).
   	$$
   	Here the interpretation is that~$\esssup_A \psi = -\infty$ if~$A = \emptyset$. If~$M$ is finite, then~$u$ is essentially bounded from above in~$B_{r/2} \equiv B_{r/2}(\xi_0)$ and
   	\begin{equation}\label{sec4_thm1_e0}
   		\esssup_{B_{r/2}}(u-m)_+ \leq \delta \ {\rm Tail}((u-m)_+;\xi_0,r/2)+ \c \  \delta^{-\gamma} \left( \, \dashint_{B_r} (u-m)_+^t \dxi\right)^\frac{1}{t},
   	\end{equation}
   	holds for all~$m \geq M$,~$t \in (0,p)$ and~$\delta \in (0,1]$ with constant~$\gamma = \gamma(n,p,s,t)>0$ and~$\c=\c(n,p,s,\Lambda,\|f\|_{L^\infty(B_r)})>0$.
   \end{theorem}
   \begin{proof}
   	Let us suppose that~$M <\infty$ and let us choose~$k \geq 0$, $m \geq M$ and a test function~$\p \in C^\infty_0(B_r)$, such that~$ 0 \leq \p \leq 1$. 
    Denote with~$v := u-(u-m-k)_+\p^p \in \mathcal{K}_{h,g}(\Omega,\Omega')$. Since~$u$ solves the obstacle problem, denoting with~$u_m := u-m$, we have that
   	\begin{align*}
   		0 & \leq \int_{\h^n}\int_{\h^n} \frac{L(u(\xi),u(\eta))(v(\xi)-u(\xi)-v(\eta)+u(\eta))}{ \dd(\eta^{-1}\circ \xi)^{Q+sp}}\dxieta\\
   		&	\quad -\int_{\h^n}f(\xi,u)(v(\xi)-u(\xi))\dxi \\
   		& = -\int_{\h^n}\int_{\h^n} \frac{L(u_m(\xi),u_m(\eta))((u_m(\xi)-k)_+\p^p(\xi)-(u_m(\eta)-k)_+\p^p(\eta))}{\dd(\eta^{-1}\circ \xi)^{Q+sp}}\dxieta\\
   		& \quad +\int_{\h^n}f(\xi,u)(u_m(\xi)-k)_+\p^p(\xi) \dxi,
   	\end{align*}
   	which yields
   	\begin{align*}
   		0 & \geq \int_{\h^n}\int_{\h^n} \frac{L(u_m(\xi),u_m(\eta))((u_m(\xi)-k)_+\p^p(\xi)-(u_m(\eta)-k)_+\p^p(\eta))}{\dd(\eta^{-1}\circ \xi)^{Q+sp}}\dxieta\\
   		& \quad -\|f\|_{L^\infty(B_r)}\int_{\h^n}(u_m(\xi)-k)_+\p^p(\xi) \dxi.
   	\end{align*}
   	As carefully explained in~\cite{MPPP21} the inequality above is enough to prove a Caccioppoli-type estimate with tail as Theorem~\ref{teo_caccioppoli}, which subsequently can be used to prove a local boundedness result like Theorem~\ref{teo_bdd}, which yields
   	\begin{equation}\label{sec4_thm1_e1}
   		\esssup_{B_{\varrho/2}(\eta)}(u_m)_+ \leq \tilde{\delta}\, {\rm Tail}((u_m)_m;\eta,\varrho/2)+\c\, \tilde{\delta}^{-\tilde{\gamma}} \left( \, \dashint_{B_\varrho(\eta)} (u_m)_+^p \dxi\right)^\frac{1}{p},
   	\end{equation}
   	where~$B_\varrho(\eta) \subset B_r$,~$\tilde{\delta} \in (0,1]$, the constant~$\tilde{\gamma}= \tilde{\gamma}(n,p,s)>0$ and~$\c$ depending on~$n$,$p$,$s$, $\Lambda$ and the function~$f$. 
   	
   	Now we apply a covering argument. Let us set
   	$$
   	\varrho= (\sigma-\sigma')r, \qquad \frac{1}{2} \leq \sigma' <\sigma \leq 1, \qquad \mbox{and} \qquad \eta \in  B_{\sigma' r}.
   	$$
   	We then proceed to estimate the nonlocal contribution in~\eqref{sec4_thm1_e1}. Let us first note that for any~$\xi \in \h^n \smallsetminus B_{\sigma r}$ we have that, by the choice of~$\eta$, it holds
   	\begin{equation}\label{sec4_thm1_e2}
   		\frac{|\xi^{-1}\circ \xi_0|_{\h^n}}{|\xi^{-1}\circ \eta|_{\h^n}} \leq 1+ \frac{|\eta^{-1}\circ \xi_0|_{\h^n}}{|\xi^{-1}\circ \xi_0|_{\h^n}-|\eta^{-1}\circ \xi_0|_{\h^n}} \leq 1+ \frac{\sigma'r}{\sigma r-\sigma'r}= \frac{\sigma r}{(\sigma-\sigma')r}  
   	\end{equation}
   	Then, by inequality~\eqref{sec4_thm1_e2} we have that the nonlocal tail on the righthand side of~\eqref{sec4_thm1_e1} can be treated as follows
   	\begin{align}\label{sec4_thm1_e3}
   		{\rm Tail}(&(u_m)_+;\eta,\varrho/2)^{p-1} \notag\\
   		& \leq \left(\frac{\varrho}{2}\right)^{sp}\sup_{B_{\sigma r}}(u_m)^{p-1}_+\int_{B_{\sigma r}\smallsetminus B_\varrho(\eta)}|\xi^{-1}\circ \eta|_{\h^n}^{-Q-sp}\dxi \notag\\
   		& \quad + \c \, (\sigma-\sigma')^{-Q} \, {\rm Tail}((u_m)_+;\xi_0,\sigma r)^{p-1}\notag\\
   		& \leq \c \, \sup_{B_{\sigma r}}(u_m)^{p-1}_+ +\c \, (\sigma-\sigma')^{-Q} \, {\rm Tail}((u_m)_+;\xi_0,r/2)^{p-1}.
   	\end{align}
   	Let us study the local contribution in~\eqref{sec4_thm1_e1}. Applying Young's Inequality, with exponent~$p/(p-t)$ and~$p/t$, we obtain that
   	\begin{align}\label{sec4_thm1_e4}
   		&\tilde{\delta}^{-\tilde{\gamma}} \left( \, \dashint_{B_\varrho(\eta)} (u_m)_+^p \dxi\right)^\frac{1}{p} \notag\\
   		& \quad \leq \tilde{\delta}^{-\tilde{\gamma}} \sup_{B_\varrho(\eta)}(u_m)_+^\frac{p-t}{p}\left( \, \dashint_{B_\varrho(\eta)} (u_m)_+^t \dxi\right)^\frac{1}{p}\notag\\
   		& \quad \leq \frac{1}{4} \sup_{B_{\sigma r}}(u_m)_+ + \c \ \tilde{\delta}^{-\frac{\tilde{\gamma}p}{t}} \left( \, \dashint_{B_\varrho(\eta)} (u_m)_+^t \dxi\right)^\frac{1}{t}\notag\\
   		& \quad \leq \frac{1}{4} \sup_{B_{\sigma r}}(u_m)_+ + \c \ \tilde{\delta}^{-\frac{\tilde{\gamma}p}{t}}  (\sigma-\sigma')^{-\frac{Q}{t}} \, \left( \, \dashint_{B_r} (u_m)_+^t \dxi\right)^\frac{1}{t}.
   	\end{align}
   	Thus, by combining~\eqref{sec4_thm1_e1} with~\eqref{sec4_thm1_e3} and~\eqref{sec4_thm1_e4} with~$\tilde{\delta} \leq 1/4\c$, we eventually arrive at
   	\begin{align}
   		\sup_{B_{\sigma'r}}(u_m)_+ & \leq \frac{1}{2}\sup_{B_{\sigma r}}(u_m)_+ + \c\,\tilde{\delta}^{-\frac{\tilde{\gamma}p}{t}}  (\sigma-\sigma')^{-\frac{Q}{t}} \, \left( \, \dashint_{B_r} (u_m)_+^t \dxi\right)^\frac{1}{t}\\
   		& \quad +\c \, \tilde{\delta}(\sigma-\sigma')^{-\frac{Q}{p-1}} \, {\rm Tail}((u_m)_+;\xi_0,r/2).\notag
   	\end{align}
   Extending the same argument of~\cite[Theorem~2]{KKP16} and taking in consideration the datum~$f$ and the non-Euclidean structure of our problem we finally obtain the desired inequality~\eqref{sec4_thm1_e0}. 
   \end{proof}
 
   \begin{theorem}\label{sec4_thm2}
   	Let~$s \in (0,1)$,~$p \in (1,\infty)$ and let us suppose that~$h$ is locally H\"older continuous in~$\Omega$ or~$h \equiv -\infty$. Then, under assumptions~\eqref{reg f},~\eqref{bound f} and~\eqref{monotonicity}, the solution~$u$ to the obstacle problem in~$\mathcal{K}_{g,h}(\Omega,\Omega')$ is locally H\"older continuous in~$\Omega$ as well. Moreover, for every~$\xi_0 \in \Omega$ there is~$r_0 >0$ such that, for any~$B_r \equiv B_r(\xi_0) \Subset B_{r_0} \equiv B_{r_0}(\xi_0)$, it holds, for any~$\varrho \in (0,r/3]$, 
   	\begin{align}\label{sec4_thm2_e0}
   		\osc_{B_\varrho} u & \leq \c \, \left(\frac{\varrho}{r}\right)^\alpha \left[ {\rm Tail}(u-h(\xi_0);\xi_0,r)+ \left( \, \dashint_{B_r} |u-h(\xi_0)|^p \dxi\right)^\frac{1}{p}\right]\\
   		& \quad +\c \int_\varrho^r \left(\frac{\varrho}{\tau}\right)^\alpha \left[ \omega_h\left(\frac{\tau}{\sigma}\right) +\varPsi\left(\frac{\tau}{\sigma}\right)\right] \frac{{\rm d }\tau}{\tau}\notag,
   	\end{align}
   where~$\omega_h(\varrho):= \osc_{B_\varrho}h$, the function~$\varPsi$ is given by
    \begin{equation}\label{chi}
  	\varPsi(\tau) = 
  	\begin{cases}
  		\tau^\frac{Qsp}{t(Q-sp)}\|f\|_{L^\infty(B_r)}^\frac{Q}{t(Q-sp)} \quad \text{for any} \ t<\min\big(p,\frac{Q(p-1)}{Q-sp}\big) \ \text{if} \ sp<Q,\\[0.8ex]
  		\tau^\frac{Q(s-\epsilon)}{t\epsilon}\|f\|_{L^\infty(B_r)}^\frac{\epsilon}{t\epsilon}  \quad \text{for any} \ 0< \epsilon<s \ \text{and}\ \text{any} \ t<\min\big(p,\frac{(p-1)s}{\epsilon}\big)  \ \text{if} \ sp=Q,
  	\end{cases}
  \end{equation}
  and~$\alpha$,~$\sigma$ and~$c$ are positive constants depending only on~$n$,~$p$,~$s$ and the structural constant~$\Lambda$ in~\eqref{def_lambda}.
   \end{theorem}

   \begin{proof}
   	We extend the same approach used in~\cite{KKP16} taking in consideration the presence of the datum~$f$ and the underlying structure given by the Heisenberg group.
   	
   	With no loss of generality we assume that~$sp \leq Q$, since if the reverse inequality holds true we have that~$u$ is H\"older continuous, as proven in~\cite{AM18}.
	
   	We divide the proof into two cases: when~$\xi_0$ belongs to the contact set, i.~\!e. when it is such that for every~$r \in (0,R)$, with~$R :=\dist(\xi_0,\partial \Omega)$, we have for~$B_r \equiv B_r(\xi_0)$,
   	$$
   	\essinf_{B_r}(u-h)=0,
   	$$
   	and when~$\xi_0$ does not belongs to the contact set.
   	
   	Note that, in this second scenario, we can find~$r_0 \in (0,R)$ such that
   	$$
   	\essinf_{B_{r_0}}(u-h)>0.
   	$$ 
   	By Corollary~\ref{sec3.1_corol1} it holds that~$u$ is a weak solution to~\eqref{problema2} in $B_{r_0}$. Hence, an application of Theorem~\ref{teo_holder} yields 
   	$$
   	\osc_{B_\varrho} u \leq \c\left(\frac{\varrho}{r}\right)^\alpha \left({\rm Tail}(u;\xi_0,r/2)+ \left(\, \dashint_{B_r}|u|^p \dxi\right)^\frac{1}{p}\right),
   	$$
   	for any~$ r \in (0,r_0)$, $\varrho \in (0,r/2)$. Thus, the desired inequality~\eqref{sec4_thm2_e0} follows. 
   	
    Let us fix~$\xi_0$ on the contact set. We want to show that, for any~$r \in (0,R)$, there exist two positive constants~$\sigma \in (0,1)$ and~$\c$ such that, for~$B_r \equiv B_r(\xi_0)$,
   	\begin{align}\label{sec4_thm2_e1}
   		\osc_{B_{\sigma r}} u & +{\rm Tail}(u-h(\xi_0);\xi_0,\sigma r)\notag\\
   		& \leq \frac{1}{2}\left(\osc_{B_r}u + {\rm Tail}(u-h(\xi_0);\xi_0,r)\right)+\c \, \omega_h(r) +\c \, \varPsi(r).
   	\end{align}
   with~$\varPsi(\cdot)$ defined in~\eqref{chi}. Let us then begin observing that~$u \geq d:= h(\xi_0)-\omega_h(r)$ almost everywhere in~$B_r$. So let us set~$u_d := u-d$. 
   By Theorem~\ref{t1 obst existence} we have that~$u_d$ is a nonnegative weak supersolution in~$B_r$ with datum~$f_d(\xi,\cdot):= f(\xi,\cdot +d)$.
    We can proceed as in Theorem~\ref{sec4_thm1} to derive the boundedness estimate~\eqref{sec4_thm1_e0} for~$u_d$, which, with~$m = d+2\omega_h(r) \geq \sup_{B_{2\varrho}}h$, gives
    \begin{equation}\label{1}
    	\sup_{B_\varrho}u_d \leq 2\omega_h(r)+\delta \ {\rm Tail}((u_d)_+;\xi_0,\varrho)+ \c \  \delta^{-\gamma} \left( \, \dashint_{B_{2\varrho}} u_d^t \dxi\right)^\frac{1}{t},
   \end{equation}	
    for~$\varrho \in (0,r]$,~$t \in (0,p)$ and~$\delta \in (0,1] $.  Moreover, since~$u_d$ is a nonegative weak supersolution of~\eqref{problema2} with datum~$f_d$ and~$\|f_d\|_{L^\infty(B_r)}=\|f\|_{L^\infty(B_r)}$, Theorem~\ref{weak harnack} yields, exchanging the role of~$f_d$ and~$f$, 
    $$
    	\left( \ \dashint_{B_{2\varrho}}  u_d^t\,{\rm d}\xi\,\right)^\frac{1}{t} \ \leq \ c\,\inf_{B_{3\varrho}} u_d 
    	\, +c \left(\frac{\varrho}{r}\right)^\frac{sp}{p-1} {\rm Tail}((u_d)_-;\xi_0,r) \, + c \, \varPsi(\varrho),
    $$
    where~$\varPsi \equiv \varPsi(\varrho)$ is defined in~\eqref{chi} and given by Theorem~\ref{weak harnack},~$\varrho \in (0,r/3)$ and~$c$ depending on~$n$,~$p$,~$s$ and the structural constant~$\Lambda$. Recalling that~$\inf_{B_\varrho}u_d \leq \omega_h(r)$, due to~$\essinf_{B_r}(u-h)=0$, and that~$u_d \geq 0$ on~$B_r$, combining the previous estimates we arrive at
   	\begin{align*}
   		\osc_{B_\varrho}u & \leq \c \, \delta^{-\gamma} \omega_h(r) +\c \, \delta {\rm Tail}(u_d;\xi_0,\varrho) \\
   		& \quad +\c \, \delta^{-\gamma} \left(\frac{\varrho}{r}\right)^\frac{sp}{p-1} {\rm Tail}(u_d;\xi_0,r)+\c \, \delta^{-\gamma} \varPsi(\varrho),
   	\end{align*}
   for~$\varrho \in (0,r/3)$, $\delta \in (0,1)$,~$\c = \c(n,p,s,\Lambda)>0$.
   
   Observe now that
   	\begin{equation}\label{sec4_thm2_e3}
   		{\rm Tail}(u_d;\xi_0,\varrho) \leq \c \, \sup_{B_r}|u_d| + \c \, \left(\frac{\varrho}{r}\right)^\frac{sp}{p-1} {\rm Tail}(u_d;\xi_0,r),
   	\end{equation}
   	and we can also estimate
   	\begin{equation}\label{sec4_thm2_e4}
   		\sup_{B_r}|u_d| =  \sup_{B_r}|u-h(\xi_0)+ \omega_h(r)| \leq \osc_{B_r}u + 2\omega_h(r).
   	\end{equation}
   	Hence, putting together all the estimates above we eventually obtain 
   	\begin{align*}
   		\osc_{B_\varrho}u & \leq \c \, \delta \osc_{B_r} u+\c \, (\delta^{-\gamma}+2\delta) \omega_h(r)\\*
   		& \quad  + \c \, \delta^{-\gamma} \left(\frac{\varrho}{r}\right)^\frac{sp}{p-1} {\rm Tail}(u_d;\xi_0,r)\\*
   		& \quad +\c \, \delta^{-\gamma} \varPsi(\varrho).
   	\end{align*}
   	For any~$\e \in (0,1)$ we choose~$\delta$ and~$\tilde{\sigma} \in (0,1)$ such that
   	$$
   	\c \, \delta \leq \frac{\e}{2}, \qquad  \mbox{and} \qquad \c \, \delta^{-\gamma} \tilde{\sigma}^\frac{sp}{p-1} \leq \frac{\e}{2}.
   	$$
   	Thus, for~$\varrho := \tilde{\sigma}r$ it holds
   	\begin{equation}\label{sec4_thm2_e5}
   		\osc_{B_{\tilde{\sigma}r}}u  \leq \e \left( \osc_{B_r} u + {\rm Tail}(u-h(\xi_0);\xi_0,r) \right) +\c \,\omega_h(r) + \c \, \varPsi(r)
   	\end{equation}
  
   	Applying estimate~\eqref{sec4_thm2_e5} with~\eqref{sec4_thm2_e3}, we can estimate for any~$\sigma \in (0,\tilde{\sigma}$)
   	\begin{align*}
   		{\rm Tail}(u-h(\xi_0);\xi_0,\sigma r) & \leq \c \, \osc_{B_{\tilde{\sigma}r}} u + \c \, \left(\frac{\sigma}{\tilde{\sigma}}\right)^\frac{sp}{p-1}{\rm Tail}(u-h(\xi_0);\xi_0,\tilde{\sigma}r)+ \c \, \omega_h(r)\\
   		& \leq \c \,\e \left( \osc_{B_r} u + {\rm Tail}(u-h(\xi_0);\xi_0,r) \right) +\c \, \omega_h(r) + \c \, \varPsi(r)\\
   		& \quad +\c \,\left(\frac{\sigma}{\tilde{\sigma}}\right)^\frac{sp}{p-1}\left( \osc_{B_r}u + {\rm Tail}(u-h(\xi_0);\xi_0,r)\right).
   	\end{align*}
   	Adding now~\eqref{sec4_thm2_e5} and choosing~$\sigma$ and~$\e$ such that
   	$$
   	\c \left(\frac{\sigma}{\tilde{\sigma}}\right)^\frac{sp}{p-1} \leq \e \quad \mbox{and} \quad (\c+2)\e \leq \frac{1}{2},
   	$$
   	we obtain~\eqref{sec4_thm2_e1}.
   	
   	Iterating~\eqref{sec4_thm2_e1} we have that, for any~$k \in \mathbb{N}$, it holds
   	\begin{align}
   		\osc_{B_{\sigma^k r}} u & + {\rm Tail}(u-h(\xi_0);\xi_0,\sigma^k r)\notag\\
   		& \leq  2^{1-k} \left( \osc_{B_{\sigma r}} u + {\rm Tail}(u-h(\xi_0);\xi_0,\sigma r)\right)\\
   		& \quad + \c \sum_{j=0}^{k-2}2^{-j}\omega_h(\sigma^{k-j-1}r) +\c \, \sum_{j=0}^{k-2}2^{-j}\varPsi(\sigma^{k-j-1}r).\notag
   	\end{align}
   	After noticing that~$\osc_{B_r} u = \osc_{B_r}u_d \leq \sup_{B_r} u_d$ and using the supremum estimate~\eqref{1}, the inequality above yields 
   	\begin{align}\label{sec4_thm2_e6}
   		\osc_{B_{\sigma^k r}} u & + {\rm Tail}(u-h(\xi_0);\xi_0,\sigma^k r)\notag\\
   		& \leq  \c \,2^{1-k} \left( {\rm Tail}(u-h(\xi_0);\xi_0,r)+ \left(\,\dashint_{B_r}|u-h(\xi_0)|^t \dxi\right)^\frac{1}{t} \, \right)\\
   		& \quad + \c \sum_{j=0}^{k-1}2^{-j}\omega_h(\sigma^{k-j-1}r) +\c  \sum_{j=0}^{k-1}2^{-j}\varPsi(\sigma^{k-j-1}r).\notag
   	\end{align}
   	 Thus, estimate~\eqref{sec4_thm2_e0} follows in the case of the contact set for some exponent~$\alpha \leq -\log 2/\log \sigma$ and some easy manipulation.
    \end{proof}
   
   The main consequence of the theorem above is the following corollary whose proof follows directly from the one in~\cite[Theorem~4]{KKP16} using~\eqref{sec4_thm2_e0}.
   
   \begin{corol}\label{sec4_corol1}
   Let~$s \in (0,1)$ and~$p \in (1,\infty)$ and let us suppose that~$h$ is  continuous in~$\Omega$ or~$h \equiv -\infty$. Then, under assumptions~\eqref{reg f},~\eqref{bound f} and~\eqref{monotonicity}, the solution to the obstacle problem in~$\mathcal{K}_{g,h}(\Omega,\Omega')$ is continuous in~$\Omega$ as well.
   \end{corol}

   \subsubsection{Boundary regularity}
    We consider the boundary behaviour of the solution~$u$ to the obstacle problem in~$\mathcal{K}_{g,h}(\Omega,\Omega')$.
 
   We begin stating some helpful lemmas. The first one is a Caccioppoli-type inequality whose proof is a verbatim repetition of the one presented in~\cite[Theorem~1.3]{MPPP21}, whose proof strategy goes back to~\cite{DKP14,DKP16}.
   \begin{lemma}\label{sec5_lem1}
   	Let  $s \in (0,1)$, $p \in (1,\infty)$, and, under assumptions~\eqref{reg f},~\eqref{bound f} and~\eqref{monotonicity}, let $u \in \mathcal{K}_{g,h}(\Omega,\Omega')$ be the solution to the obstacle problem in~$\mathcal{K}_{g,h}(\Omega,\Omega')$. Let~$\xi_0 \in \partial \Omega$ and let~$r \in (0,\dist(\xi_0,\partial \Omega'))$, and suppose that, for~$B_r \equiv B_r(\xi_0)$,
   	$$
   	k_+ \geq \max\left(\esssup_{B_r}g, \esssup_{B_r\cap \Omega}h\right) \quad \mbox{and} \quad k_- \leq \essinf_{B_r} g.
   	$$
   	Then, for~$w_\pm:=(u-k_\pm)_\pm$, we have
   	\begin{align}
   		\int_{B_r} \int_{B_r} & |\eta^{-1} \circ \xi|_{\h^n}^{-Q-sp}  |w_\pm(\xi)\p(\xi)-w_\pm(\eta)\p(\eta)|^p \, \dxieta\notag\\*
   		& \leq \c\int_{B_r}\int_{B_r} | \eta^{-1} \circ \xi|_{\h^n}^{-Q-sp} w_\pm^p(\xi)|\p(\xi)-\p(\eta)|^p \, \dxieta\\*
   		&\quad+\c \int_{B_r}w_\pm(\xi)\p^p(\xi) \, \,{\rm d}\xi \left(\sup_{\eta \in \textup{supp}\, \p}\int_{\h^n \smallsetminus B_r} |\eta^{-1} \circ \xi|_{\h^n}^{-Q-sp} w_\pm^{p-1}(\xi) \,\dxi\right.\notag   \\*
   		& \hspace{4.5cm} \left.  + \, \|f\|_{L^\infty(B_r)}\right)\notag,
   	\end{align}
   	for any~$\p \in C^\infty_0(B_r)$ with~$0 \leq \p \leq 1$ and $c$ denepnding on $n$, $p$, $s$ and the structural constant $\Lambda$.
   \end{lemma}
  
   We remark that if~$\max(\esssup_{B_r}g, \esssup_{B_r\cap \Omega}h)$ is infinite or~$\essinf_{B_r} g = - \infty$ then the interpretation is that no test function of the type~$w_\pm$ exists.
   
   \vspace{2mm}
   Let us note that since we have a Caccioppoli-type estimate on the boundary of the set~$\Omega$ we can prove a local boundedness estimate following the same iteration argument already used for the proof of Theorem~\ref{teo_bdd}.
   
   \begin{theorem}\label{sec5_thm1_boundedness}
   	Suppose that, under assumptions~\eqref{reg f},~\eqref{bound f} and~\eqref{monotonicity},~$u \in \mathcal{K}_{g,h}(\Omega,\Omega')$ solves the obstacle problem in~$\mathcal{K}_{g,h}(\Omega,\Omega')$ and let~$ s \in (0,1)$ and~$p \in (1,\infty)$. Let~$	\xi_0 \in \partial \Omega$ and~$r \in (0,\dist(\xi_0,\partial \Omega)$. Assume that, for~$B_r \equiv B_r(\xi_0)$, 
   	$$
   	\esssup_{B_r}g + \esssup_{B_r \cap \Omega} h <\infty \quad \mbox{and} \quad \essinf_{B_r} g > -\infty.
   	$$
  Then,~$u$ is essentially bounded close to~$\xi_0$.
   \end{theorem}
   In order to prove the H\"older continuity of the solution to the obstacle problem on the boundary a logarithmic-type estimate is needed. 
   Then, following the same argument of~\cite[Lemma~1.4]{MPPP21}, we have 

   \begin{lemma}\label{sec5_lem3}
   	Suppose that, under assumptions~\eqref{reg f},~\eqref{bound f} and~\eqref{monotonicity}, ~$u \in \mathcal{K}_{g,h}(\Omega,\Omega')$ solves the obstacle problem in~$\mathcal{K}_{g,h}(\Omega,\Omega')$ with~$s \in (0,1)$ and~$p \in ( 1,\infty)$. Let $B_R \equiv B_R(\xi_0) \Subset \Omega'$, $B_r \subset B_{R/2}$ and assume that
   	$$
   	+\infty > k_+ \geq \max \left( \esssup_{B_R}g, \esssup_{B_R\cap \Omega}h\right) \quad \mbox{and} \quad -\infty < k_- \leq \essinf_{B_R}g.
   	$$
   	Then the functions
   	$$
   	w_\pm := \esssup_{B_R}(u-k_\pm)_\pm -(u-k_\pm)_\pm +\e
   	$$
   	satisfy the following estimate
   	\begin{align}
   		\int_{B_r}&\int_{B_r}  \left|\log \frac{w_\pm(\xi)}{w_\pm(\eta)}\right|^p \dd(\eta^{-1} \circ \xi)^{-Q-sp} \dxieta \notag\\
   		& \leq \c \, r^{Q-sp} \left(1+ r^{sp}\e^{1-p}\|f\|_{L^\infty(B_R)}+\e^{1-p} \left(\frac{r}{R}\right)^{sp} {\rm Tail}((w_\pm)_-;\xi_0,R)^{p-1}\right),
   	\end{align}
   	for every~$\e >0$ and $c$ denepnding on $n$, $p$, $s$ and the structural constant $\Lambda$.
   \end{lemma}

    We can now prove the following result.
    \begin{lemma}\label{sec5.1_lem2}
   	Assume that~$\xi_0 =0\in \partial \Omega$ and~$g(0)=0$, where~$\Omega$ satisfies~\eqref{condizione misura} for all~$r \leq R$. Suppose that, under assumptions~\eqref{reg f},~\eqref{bound f} and~\eqref{monotonicity},~$u \in \mathcal{K}_{g,h}(\Omega,\Omega')$ solves the obstacle problem in~$\mathcal{K}_{g,h}(\Omega,\Omega')$ with~$s \in (0,1)$ and~$p \in (1,\infty)$. Let~$\omega>0$. Then, there exist~$\tau_0 \in (0,1)$, $\sigma \in (0,1)$ and $\vartheta \in (0,1)$, all depending on~$n$,~$p$,~$s$ and~$\delta_\Omega$, such that if
   	\begin{equation}\label{sec5.1_lem2_e0}
   		\osc_{B_R} u + \sigma \, {\rm Tail}(u;0,R) \leq \omega, \quad \osc_{B_R}g \leq \frac{\omega}{8}, \quad \mbox{and} \quad \|f\|_{L^\infty (B_R)} \leq (\vartheta \omega)^{p-1},
   	\end{equation}
   	hold for~$B_R \equiv B_R(0)$, then the following decay estimate 
   	\begin{equation}\label{sec5.1_lem2_e1}
   		\osc_{B_{\tau R}} u + \sigma \, {\rm Tail}(u;0,\tau R) \leq (1-\vartheta)\omega ,
   	\end{equation}
   	holds as well for every~$\tau \in (0,\tau_0]$.
   \end{lemma}
   
   We will use the following lemma, whose proof can be done extending the one presented in~\cite[Lemma~7]{KKP16} and using the Poincaré inequality in Proposition~\ref{poincare}.
   \begin{lemma}\label{sec5.1_lem1}
   	Let us consider a set~$\Omega$ which satisfies condition~\eqref{condizione misura} for~$r_0>0$ and~$\delta_\Omega >0$ and let $B_r \equiv B_r(\xi_0)$ with~$\xi_0 \in \partial \Omega$ and~$r \in (0,r_0)$. Let us suppose that~$\psi \in W^{s,p}(B_r)$ and~$\psi \equiv 0$ in~$B_r \smallsetminus \Omega$. Then
   	\begin{equation}\label{sec5.1_lem1_e0}
   		\dashint_{B_r}|\psi|^p \dxi \leq \c \left(1-(1-\delta_\Omega)^{1-1/p}\right)^{-p}r^{sp} \int_{B_r} \ \dashint_{B_r} \frac{|\psi(\xi)-\psi(\eta)|^p}{|\eta^{-1}\circ \xi|_{\h^n}^{Q+sp}}\dxieta.
   	\end{equation}
   \end{lemma}

   \begin{proof}[Proof of Lemma~{\rm \ref{sec5.1_lem2}}]
   	Let us begin denoting with~$H = \vartheta /\sigma$ and~$B \equiv B_R(0)$. We can estimate the tail term as follows
   	\begin{align*}
   		\sigma^{p-1}{\rm Tail}(u;0,\tau R)^{p-1} & = \sigma^{p-1} (\tau R)^{sp} \int_{B \smallsetminus \tau B}\frac{|u(\xi)|^{p-1}}{|\xi|_{\h^n}^{Q+sp}} \dxi\\
   		& \quad + \sigma^{p-1}\tau^{sp}{\rm Tail}(u;0,R)^{p-1}\\
   		& \leq \c \,\sigma^{p-1}\omega^{p-1} + \tau^{sp}\omega^{p-1}.
   	\end{align*}
   	Hence, 
   	\begin{equation}
   		\sigma{\rm Tail}(u;0,\tau R) \leq \c \left(\frac{\vartheta}{H} + \tau^\frac{sp}{p-1}\right)\omega \leq \frac{2\c\vartheta}{H}\omega \leq \vartheta \omega
   	\end{equation} 
   	where we have chosen~$\tau_0 \leq \sigma^{(p-1)/(sp)}$ and~$H = 2\c $. Thus, we only have to show
   	\begin{equation}\label{sec5.1_lem2_e3}
   		\osc_{\tau B} u \leq (1-2\vartheta)\omega, \qquad \forall \tau \leq \tau_0.
   	\end{equation}
   	Let us adopt the following notation
   	$$
   	k_+ := \sup_B u -\frac{\omega}{4}, \quad, k_- := \inf_B u + \frac{\omega}{4}, \quad \e:= \vartheta \omega,
   	$$
   	and
   	$$
   	w_\pm := \sup_{B}(u-k_\pm)_\pm -(u-k_\pm)_\pm +\e, \quad \tilde{w}_\pm := \frac{w_\pm}{\sup_{B}w_\pm}.
   	$$
   	Clearly we can assume that~$\sup_B u \geq \frac{3}{8}\omega$ and~$\inf_B u \leq -\frac{3}{8}\omega$ since if the reverse inequalities hold true we have~$\osc_B u \leq \frac{3}{4}\omega$ and~\eqref{sec5.1_lem2_e3} follows choosing~$\vartheta \leq 1/8$.
   	
   	We consider the case~$\sup_B u \geq \frac{3}{8}\omega$ being the other one symmetric. Let us note that due to the condition~$u = g \leq \omega/8$ in~$B \smallsetminus \Omega$ then $\tilde{w}_+ = 1$ in $B \smallsetminus \Omega$. Hence, we apply Lemma~\ref{sec5.1_lem1}, Proposition~\ref{prop1} and Lemma~\ref{sec5_lem3} together with~\eqref{sec5.1_lem2_e0}, with $r = 2\tau R$ and $\tau_0 \leq \min(1/4 , \sigma^{2(p-1)/sp})$, obtaining
   	\begin{align*}
   		\dashint_{2\tau B} |\log \tilde{w}_+|^p \dxi & \leq \c \, (\tau R)^{sp}\int_{2\tau B} \ \dashint_{2\tau B}\left| \, \log \frac{\tilde{w}_+(\xi)}{\tilde{w}_+(\eta)}\right|^p \dd(\eta^{-1}\circ\xi)^{-Q-sp} \dxieta\\
   		& \leq \c \, \left(1+ (2\tau R)^{sp}\e^{1-p}\|f\|_{L^\infty(2\tau B)}+(\vartheta \omega)^{1-p} \left(2\tau\right)^{sp} {\rm Tail}((\tilde{w}_+)_-;0,R)^{p-1}\right)\\*
      & \leq \c	
   	\end{align*}
   	Combining Chebyshev's Inequality  with the previous estimate we get
   	\begin{align}\label{sec5.1_lem2_e4}
   		\frac{|2\tau B \cap \{|\log \tilde{w}_+| \geq |\log(20\vartheta)|\}|}{|2\tau B|}  &\leq  |\log (20 \vartheta)|^{-p} \dashint_{2\tau B}|\log \tilde{w}_+|^p \dxi\notag\\
   		&   		\leq \c \, |\log(20\vartheta)|^{-p}.
   	\end{align}
   	We want to estimate now the left-hand side of~\eqref{sec5.1_lem2_e4}. Recall that, by definition~$0 < \tilde{w}_+ \leq 1$ and~$\sup_B (u-k_+)_+ = \omega/4$, we obtain that, when~$\vartheta < 1/20$,
   	\begin{align*}
   		\{|\log \tilde{w}_+| \geq |\log(20\vartheta)|\} & = \{\tilde{w}_+ \leq 20\vartheta\}\\
   		& = \left\{ \frac{\omega}{4}-(u-k_+)_+ +\e \leq 20 \vartheta \left(\frac{\omega}{4}+\e\right)\right\} 
   		 \supset \big\{u \geq \sup_{B}u -4\vartheta\omega\big\}.
   	\end{align*}
   	Hence, denoting with~$\tilde{k}:= \sup_{B}u -4\vartheta\omega$ we eventually arrive at
   	\begin{align*}
   		\left(\,\dashint_{2\tau B}(u-\tilde{k})_+^p\dxi\,\right)^\frac{1}{p} & \leq 4\vartheta\omega \left(\frac{|2\tau B \cap\{u \geq \sup_{B}u-4\vartheta\omega\}|}{|2\tau B|}\right)^\frac{1}{p}
   		 \leq \frac{\c \vartheta\omega}{|\log(20\vartheta)|}.
   	\end{align*}
   	Since~$\tilde{k} \geq \sup_{B}g$ by Theorem~\ref{sec5_thm1_boundedness} we obtain that
   	$$
   	\sup_{\tau B}(u-\tilde{k})_+ \leq \delta \, {\rm Tail}((u-\tilde{k})_+;0,\tau R)+\c \, \delta^{-\gamma}\left( \  \dashint_{2\tau B}(u-\tilde{k})_+^p \dxi\right)^\frac{1}{p},
   	$$
   	for any~$\delta \in (0,1]$. Thus,
   	\begin{equation}\label{sec5.1_lem2_e5}
   		\sup_{\tau B} u \leq \sup_B u -4\vartheta\omega +\delta {\rm Tail}((u-\tilde{k})_+;0,\tau R)  + \frac{\c \delta^{-\gamma} \vartheta\omega}{|\log(20\vartheta)|}.
   	\end{equation}
   	We are left to estimating the tail term
   	\begin{align*}
   		{\rm Tail}((u-\tilde{k})_+;0,\tau R)^{p-1} & \leq (\tau R)^{sp} \int_{B \smallsetminus \tau B} \frac{(u-\tilde{k})_+^{p-1}}{|\xi|_{\h^n}^{Q+sp} }\dxi\\
   		& \qquad+ \tau^{sp}{\rm Tail}((u-\tilde{k})_+;0,R)^{p-1}\\
   		& \leq \c \, (\vartheta \omega)^{p-1} \left(1+\frac{\tau^{sp}}{\vartheta^{p-1}\sigma^{p-1}}\right) \leq c(\vartheta\omega)^{p-1},
   	\end{align*}
   	where we have taken~$\tau^{sp} \leq \tau_0^{sp} \leq (\sigma\vartheta)^{p-1}$ and we have used condition~\eqref{sec5.1_lem2_e0}. Then, up to taking~$\delta$ and~$\vartheta$ sufficiently small, from~\eqref{sec5.1_lem2_e5}, we have that
   	$$
   	\sup_{\tau B} \leq \sup_{B}u -2\vartheta\omega,
   	$$
   	and condition~\eqref{sec5.1_lem2_e3} follows as desired.
   \end{proof}
   We are now able to prove our main result.
   \begin{theorem}\label{sec5.1_thm1}
   	Under assumptions~\eqref{reg f},~\eqref{bound f} and~\eqref{monotonicity} let~$u \in \mathcal{K}_{g,h}(\Omega,\Omega')$ be the solution to the obstacle problem in~$\mathcal{K}_{g,h}(\Omega,\Omega')$ with~$s \in (0,1)$ and~$p \in (1,\infty)$ and assume~$\xi_0 \in \partial \Omega$ and~$B_{2R}(\xi_0) \subset \Omega'$.  If~$g \in \mathcal{K}_{g,h}(\Omega,\Omega')$ is H\"older continuous in~$B_R(\xi_0)$ and~$\Omega$ satisfies~\eqref{condizione misura} for all~$r \leq R$, then~$u$ is H\"older continuous in~$\Omega$ as well.
   \end{theorem}
   \begin{proof}
   	With no loos of generality we may assume that~$\xi_0 =0$ and $g(0)=0$. Moreover, we can suppose that there exists~$R_0$ such that~$\osc_{B_0}g \leq \osc_{B_0}u$ for~$B_0 \equiv B_{R_0}(0)$ since otherwise the thesis follows by the H\"older continuity of~$g$.
   	
   	Let us define now, for~$\vartheta \in (0,1)$,
   	\begin{equation}\label{sec5.1_thm1_e1}
   		\omega_0 := 8 \left(\osc_{B_0}u + {\rm Tail}(u;0,R_0) +\frac{ \|f\|^\frac{1}{p-1}_{L^\infty(B_0)}}{\vartheta}\right).
   	\end{equation}
   	By Lemma~\ref{sec5.1_lem2} we have that there exist~$\tau_0$, $\sigma$ and $\vartheta$ depending only on $n$,~$p$,~$s$ and~$\delta_\Omega$ such that if
   	\begin{equation}\label{sec5.1_thm1_e2}
   		\osc_{B_r(0)} u + \sigma \, {\rm Tail}(u;0,r) \leq \omega \quad \osc_{B_r(0)}g \leq \frac{\omega}{8} \quad \mbox{and} \quad \|f\|_{L^\infty (B_r(0))} \leq \left(\vartheta \omega \right)^{p-1},
   	\end{equation}
   	hold for a ball~$B_r(0)$ and for~$\omega>0$, then
   	\begin{equation}\label{sec5.1_thm1_e3}
   		\osc_{B_{\tau r}(0)} u + \sigma \, {\rm Tail}(u;0,\tau r) \leq (1-\vartheta)\omega,
   	\end{equation}
   	holds for every~$\tau \in (0,\tau_0]$. Let us take~$\tau \leq \tau_0$ such that
   	\begin{equation}\label{sec5.1_thm1_e4}
   		\osc_{\tau^j B_0}g \leq (1-\vartheta)^j \frac{\omega_0}{8}, \quad \mbox{and} \qquad \|f\|_{L^\infty(\tau^j B_0)} \leq \big( \vartheta (1- \vartheta)^j \omega_0\big)^{p-1}, \quad \mbox{for every } j =0,1,\dots
   	\end{equation}
   	We now iterate~\eqref{sec5.1_thm1_e3} with~\eqref{sec5.1_thm1_e2} and~\eqref{sec5.1_thm1_e4}, noticing that the initial condition is satisfied if we choose~$\omega_0$ as in~\eqref{sec5.1_thm1_e1}. Thus, we obtain that
   	$$
   	\osc_{\tau^j B_0}u \leq (1-\vartheta)^j \omega_0 \quad \mbox{for every } j =0,1,\dots
   	$$
   	from which follows that~$u \in C^{0,\alpha}_{\rm loc}(B_0)$ with exponent~$\alpha = \log (1-\vartheta)/\log \tau \in (0,1)$.
   \end{proof}
   From the previous theorem we can easily show that the following result holds true as well.
   \begin{theorem}\label{sec5.1_thm2}
   Let~$s \in (0,1)$, $p \in (1,\infty)$ and, under assumptions~\eqref{reg f},~\eqref{bound f} and~\eqref{monotonicity}, let~$u$ solves the obstacle problem in~$\mathcal{K}_{g,h}(\Omega,\Omega')$. Assume~$\xi_0 \in \partial \Omega$ and $B_{2R}(\xi_0) \subset \Omega'$. If~$g \in \mathcal{K}_{g,h}(\Omega,\Omega')$ is continuous in $B_R(\xi_0)$ and~$\Omega$ satisfies condition~\eqref{condizione misura} for all~$r \leq R$, then~$u$ is continuous in~$B_R(\xi_0)$ as well.
   \end{theorem}
 
  Combining together Theorem~\ref{sec4_thm2}, Corollary~\ref{sec4_corol1}, Theorem~\ref{sec5.1_thm1} and Theorem~\ref{sec5.1_thm2} we obtain Theorem~\ref{boundary reg}.
  
  \section{Properties of fractional weak supersolutions}\label{sec_weak_supersol}
   In this section we prove that weak supersolutions to~\eqref{problema2} enjoy some properties such as boundedness estimate and comparison principle. Moreover, we give  a characterization of their pointwise behaviour and recall some helpful facts about convergence of sequences of weak supersolutions that will be used in the proof of Theorem~\ref{perron}. All the forthcoming results are of independent interest in the analysis of nonlinear fractional equations in the Heisenberg group. They are the natural extension of the ones proven in~\cite{KKP17} for the homogeneous end Euclidean case.
  
   We always assume that~$f$ satisfies conditions~\eqref{reg f},~\eqref{bound f}, and~\eqref{monotonicity}. However, in most of the forthcoming properties condition~\eqref{monotonicity} is not strictly necessary, thus most of the following propositions simply work requiring only the regularity hypothesis in~\eqref{reg f} and~\eqref{bound f}. When not explicitly specified in the statements of the results we only assume that only this regularity conditions hold true.
  
  \subsection{A priori bounds for weak supersolutions}
  We show that weak supersolutions satisfy a Caccioppoli-type inequality with tail (Lemma~\ref{s2 lem2 Cacc} below). This estimate will then imply an upper bound on the fractional norm for weak supersolutions, namely Lemma~\ref{s3 bdd above seminorm}.
  
  Let us recall the following result whose proof is omitted being a plain extension of its Euclidean counterpart~\cite[Lemma~3]{KKP17} and of Theorem~\ref{teo_bdd}.
  
  \begin{lemma}\label{s2 lem1 bdd below}
  Let~$s \in (0,1)$, $p \in (1,\infty)$. Let~$u$ be a weak supersolution in~$\Omega$,~$h \in L^{p-1}_{sp}(\h^n)$ and assume that~$h \leq u \leq 0$ almost everywhere in~$\h^n$. Then, for all~$D \Subset \Omega$ there is a constant~$\c \equiv \c(n,p,s,\Lambda,\Omega,D,h)>0$ such that
  $$
  \essinf_D u \geq -\c.
  $$
  \end{lemma}
  
  From Theorem~\ref{teo_caccioppoli} we can deduce another Caccioppoli-type estimate for weak supersolutions.
  
  \begin{lemma}\label{s2 lem2 Cacc}
  Let~$s \in (0,1)$, $p \in (1,\infty)$ and~$M >0$. Suppose that~$u$ is a weak supersolution in~$B_{2r} \equiv B_{2r}(\xi_0)$ such that~$u \leq M $ in~$B_{3r/2}$. Then, for a positive constant~$\c \equiv \c(n,p,s,\Lambda)$, it holds
  \begin{equation}
  	\int_{B_r} \ \dashint_{B_r} \frac{|u(\xi)-u(\eta)|^{p}}{|\eta^{-1}\circ \xi|_{\h^n}^{Q+sp}}\dxieta \leq \c \, \max(r^{-sp},1) \, H^p + \c \,\|f\|_{L^\infty(B_{3r/2})} \, H,
  \end{equation}
  where
  $$
  H := M + \left( \ \dashint_{B_{3r/2}}u_-^p(\xi) \dxi\right)^\frac{1}{p} + \textup{Tail}(u_-;\xi_0,3r/2) +\left( \ \dashint_{B_{3r/2}} f_-^\frac{p}{p-1}(\xi,u) \dxi\right)^\frac{1}{p}.
  $$
  \end{lemma} 
   \begin{proof}
   Let us choose a cut-off function~$\p \in C^\infty_0(B_{4r/3})$ such that~$0 \leq \p \leq 1 $ and $\p \equiv 1$ in~$B_r$ and~$ |\nabla_{\h^n}\p| \leq \c /r$. Setting~$ w := 2H -u$ and choosing as test function~$\psi := w\p^p$ in Definition~\ref{solution to inhomo pbm}, we get
   \begin{align}\label{s2 lem2 1}
   & \frac{1}{|B_r|}\int_{\h^n} f( \xi,u) w(\xi)\p^p(\xi) \dxi\notag\\
   &\leq \frac{1}{|B_r|}\int_{\h^n}\int_{\h^n}\frac{L(u(\xi),u(\eta))\left(w(\xi)\p^p(\xi)-w(\eta)\p^p(\eta)\right)}{\dd(\eta^{-1}\circ \xi)^{Q+sp}}\dxieta\notag\\
   & = -\frac{1}{|B_r|} \int_{B_{3r/2}}\int_{B_{3r/2}}\frac{L(w(\xi),w(\eta))(w(\xi)\p^p(\xi)-w(\eta)\p^p(\eta))}{\dd(\eta^{-1}\circ \xi)^{Q+sp}}\dxieta\notag\\
   & \quad +\frac{2}{|B_r|}\int_{\h^n \smallsetminus B_{3r/2}}\int_{B_{3r/2}}\frac{L(u(\xi),u(\eta))w(\xi)\p^p(\xi)}{\dd(\eta^{-1}\circ \xi)^{Q+sp}}\dxieta\notag\\
   & =: -I_1+2I_2.
   \end{align}
   We estimate separately the integrals~$I_1$ and~$I_2$ in the righthand side of~\eqref{s2 lem2 1}. We start with~$I_1$. As in the proof of Theorem~\ref{teo_caccioppoli}, assuming without losing generality,~$w(\xi) \geq w(\eta)$ we can deduce that
   \begin{align}\label{s2 lem2 2}
   I_1 & \geq \frac{1}{\c}\int_{B_{3r/2}}\ \dashint_{B_{3r/2}} \frac{|u(\xi)-u(\eta)|^p}{|\eta^{-1}\circ \xi|_{\h^n}^{Q+sp}}\left(\max(\p(\xi),\p(\eta))\right)^p\dxieta\notag\\
   & \quad -\c\int_{B_{3r/2}} \ \dashint_{B_{3r/2}} (2H -u(\xi))^p\frac{|\p(\xi)-\p(\eta)|^p}{|\eta^{-1}\circ \xi|_{\h^n}^{Q+sp}} \dxieta\notag\\
   & \geq \frac{1}{\c}\int_{B_{3r/2}} \ \dashint_{B_{3r/2}} \frac{|u(\xi)-u(\eta)|^p}{|\eta^{-1}\circ \xi|_{\h^n}^{Q+sp}}\left(\max(\p(\xi),\p(\eta))\right)^p\dxieta - \c \, r^{-sp}H^p,
   \end{align}
   where in the last inequality we have used Lemma~\ref{lem1} in order to estimate the singular integral given by the homogeneous norm~$|\cdot|_{\h^n}$ in the same spirit of Lemma~1.4 in~\cite{MPPP21}.
   
   For the second integral~$I_2$ we have
   \begin{align}\label{s2 lem2 3}
   I_2 & \leq \c\int_{\h^n \smallsetminus B_{3r/2}}\int_{B_{4r/3}}(H^{p-1}+u^{p-1}_-(\eta))(2H+u_-(\xi))|\eta^{-1}\circ \xi_0|_{\h^n}^{-Q-sp}\dxieta\notag\\
   & \leq \c \, r^{-sp}H^p +\c \, H \int_{\h^n \smallsetminus B_{3r/2}}u_-^{p-1}(\eta)|\eta^{-1}\circ \xi_0|_{\h^n}^{-Q-sp}\deta \notag\\
   & \leq \c \, r^{-sp}H^p.
   \end{align}
   We are left with estimating the term on the lefthand side of~\eqref{s2 lem2 1}
   \begin{align}\label{s2 lem2 4}
   \frac{1}{|B_r|}\int_{\h^n}f(\xi,u)w(\xi)\p^p(\xi) \dxi & \geq  -\frac{1}{|B_r|}\int_{B_{3r/2}}f_-(\xi,u)(2H+u_-(\xi)) \dxi \notag\\
   & \geq -\c \left(H \ \dashint_{B_{3r/2}}f_-(\xi,u)\dxi + \dashint_{B_{3r/2}} f_-(\xi,u)u_-(\xi) \dxi\right)\notag\\
   & \geq -\c \, H^p -\c \, \|f\|_{L^\infty(B_{3r/2})}H.
   \end{align}
   Combining~\eqref{s2 lem2 2},~\eqref{s2 lem2 3} and~\eqref{s2 lem2 4} we obtain the desired inequality.
   \end{proof}
    
    Putting together all the results above we can extend with no difficulty the proof used in~\cite[Lemma~5]{KKP17} to show a uniform bound in $W^{s,p}$ for weak supersolutions.
    
    \begin{lemma}\label{s3 bdd above seminorm}
    Let $s \in (0,1)$, $p \in (1,\infty)$, $M>0$ and $h \in L^{p-1}_{sp}(\h^n)$ with $h \leq M$. Let $u$ be a weak supersolution in $\Omega$ such that $u \geq h$ almost everywhere in $\h^n$ and $u \leq M$ almost everywhere in $\Omega$. Then, for all $D \Subset\Omega$ there is a constant $\c = \c (n,p,s,\Lambda, \Omega, D, M, h)>0$ such that
    \begin{equation}\label{s3 1}
    \int_D \int_D \frac{|u(\xi)-u(\eta)|^p}{|\eta^{-1}\circ \xi|_{\h^n}^{Q+sp}}\dxieta \leq \c.
    \end{equation}
    \end{lemma}

   \subsection{Comparison principle for weak supersolutions}
   We show now that weak supersolutions enjoy a comparison principle. This result constitutes a fundamental tool in the whole PDE theory. The proof of the forthcoming Proposition differs to the class of problem considered in~\cite{KKP17} due to the presence of the datum~$f$. In particular, as already carefully explained in the Introduction, the monotonicity hypothesis~\eqref{monotonicity} is necessary in order to derive the desired comparison principle.
   
   \begin{prop}[{\bf Comparison principle}]\label{comparison a.e.}
   Let~$\Omega$ be a bounded open subsets of~$\h^n$. Let~$u \in W^{s,p}(\Omega)$ be a weak supersolution to~\eqref{problema2} in~$\Omega$ and let~$v\in W^{s,p}(\Omega)$ be a weak subsolution to~\eqref{problema2} in~$\Omega$ such that~$ v \leq u$ almost everywhere in~$\h^n \smallsetminus \Omega$ and 
    $$
   	\limsup_{\eta \rightarrow \xi}v(\eta) \leq \liminf\limits_{\eta \rightarrow \xi} u(\eta), \qquad \mbox{for any}~\xi \in \partial \Omega.
   $$ 
   Then, under assumptions~\eqref{reg f},~\eqref{bound f} and~\eqref{monotonicity}, we have that~$v \leq u$ a.\!~e. in~$\Omega$ as well.
   \end{prop}
   \begin{proof}
   We begin noticing that there exists a compact set~$K$ such that~$\{v > u\} \Subset K \Subset \Omega$. Hence, the function~$\psi := (v-u)_+ \in W^{s,p}_0(K)$ is a proper test function. Considering the weak formulation both for~$v$ and~$u$ in Definition~\ref{solution to inhomo pbm} with~$\psi$ defined above, we obtain
  \begin{equation}\label{s3 comp 2}
  \int_{\h^n} f(\xi,v) \psi(\xi) \dxi \geq \int_{\h^n}\int_{\h^n} \frac{L(v(\xi),v(\eta))(\psi(\xi)-\psi(\eta))}{\dd(\eta^{-1} \circ \xi)^{Q+sp}}\dxieta,
  \end{equation}
  and
   \begin{equation}\label{s3 comp 3}
  	\int_{\h^n} f(\xi,u) \psi(\xi) \dxi \leq \int_{\h^n}\int_{\h^n} \frac{L(u(\xi),u(\eta))(\psi(\xi)-\psi(\eta))}{\dd(\eta^{-1} \circ \xi)^{Q+sp}}\dxieta.
  \end{equation}
  Then, summing together~\eqref{s3 comp 2} with~\eqref{s3 comp 3} we eventually arrive at
   \begin{align}\label{s3 comp 1}
   0 & \leq \int_{\{v > u\}} (f(\xi,v)-f(\xi,u))(v(\xi)-u(\xi)) \dxi\notag\\
     & \quad + \iint_{\{v > u \}} \frac{\big(L(u(\xi),u(\eta))-L(v(\xi),v(\eta))\big)\big(v(\xi)-u(\xi)-v(\eta)+u(\eta)\big)} {\dd(\eta^{-1}\circ \xi)^{Q+sp}}\dxieta\notag\\
     & \quad + 2\int_{\{v \leq u\}}\int_{\{v > u \}} \frac{\big(L(u(\xi),u(\eta))-L(v(\xi),v(\eta)\big)\big(v(\xi)-u(\xi)\big)}{ \dd(\eta^{-1}\circ \xi)^{Q+sp}}\dxieta\notag\\
     & \leq 0
   \end{align}
  By the monotonicity of the function~$L(a,b)$ and by condition~\eqref{monotonicity}.
  Thus, we obtain that all terms in~\eqref{s3 comp 1} must be~$0$. This implies that~$\psi \equiv0$ almost everywhere on~$\{v > u  \}$. Hence, we obtain that~$|\{v > u\}|=0$. 
  \end{proof}

   \subsection{Pointwise behaviour of weak supersolutions}
   Let us focus now on the pointwise behaviour of weak supersolutions. Before stating our main result, i.~\!e.~Theorem~\ref{teo lsc}, some considerations are needed. First of all, due to the presence of the datum~$f$, we could not trivially prove that the function~$v:= u(\xi_0)-u$ for any~$\xi_0 \in \Omega$ is a weak subsolution to~\eqref{problema2}, as done for the homogeneous case in~\cite{KKP17}. Nevertheless such a function satisfies a Caccioppoli-type inequality with tail which is the starting point in proving a boundedness estimate as Theorem~\ref{teo_bdd}. Inequality~\eqref{eq_bdd} will have a central role in proving the forthcoming theorem. Indeed, thanks to the its ductility, we can choose a proper interpolation parameter~$\delta \in (0,1]$ in order to carefully estimate the nonlocal contributions and the local ones. Extending the procedure used in~\cite{KKP17} we have
   
   \begin{theorem}\label{teo lsc}
   Let~$s \in (0,1)$ and~$p \in (1,\infty)$ and let~$u$ be a weak supersolution in~$\Omega$ to problem~\eqref{problema2}. Then we have
   $$
   u(\xi) := \textup{ess}\liminf_{\eta \rightarrow \xi} u(\eta), \qquad \mbox{for a.~\!e.~}\xi \in \Omega.
   $$
   In particular, $u$ has a lower semicontinuous representative.
   \end{theorem}

   \subsection{Convergence result for weak supersolutions}
      We conclude this section recalling some convergence results for sequences of weak supersolutions from the Euclidean setting. 
      This properties will turn out to be of crucial importance in the following section.

      \begin{theorem}
      Let~$s \in (0,1)$ and~$ p \in (1,\infty)$. Consider~$g,h \in L^{p-1}_{sp}(\h^n)$ such that~$h \leq g$ in~$\h^n$. Let~$\{u_j\}$ be a sequence of weak supersolutions in~$\Omega$ such that~$h \leq u_j \leq g $ almost everywhere in~$\h^n$ and~$u_j$ is uniformly essentially bounded from above in~$\Omega$. Let
      $$
      u(\xi):= \lim_{j \rightarrow \infty}u_j(\xi), \qquad \mbox{for a.~\!e.} \ \xi \in \h^n.
      $$
      Then, under assumptions~\eqref{reg f},~\eqref{bound f} and~\eqref{monotonicity},~$u$ is a weak supersolution in~$\Omega$.
      \end{theorem}
      The proof of the previous theorem is omitted since it is a verbatim repetition of the one presented in~\cite[Theorem~10]{KKP17}. Indeed, thanks to the regularity assumption made on the datum~$f$ in~\eqref{reg f} and~\eqref{bound f}, we can easily pass to the limit on the righthand side of the integral inequality in Definition~\ref{solution to inhomo pbm}. The lefthhand side can then be treated as in the Euclidean case presented in~\cite{KKP17}. As main consequence of the previous theorem we have    
      \begin{corol}\label{corol sequence weak solutions}
      Let~$s \in (0,1)$, $p \in (1,\infty)$, $h$,~$g \in L^{p-1}_{sp}(\h^n)$ and let~$\{u_j\}$ be a sequence of continuous weak solutions in~$\Omega$ such that~$h \leq u_j \leq g$ and~$u:=\lim_{j \rightarrow \infty}u_j$ exists almost everywhere in~$\h^n$. Then, under assumptions~\eqref{reg f},~\eqref{bound f} and~\eqref{monotonicity},~$u$ exists at every point of $\Omega$ and is a continuous weak solution in~$\Omega$.
      \end{corol}

     \section{The Perron method}\label{sec_perron} 
     We have everything needed to prove Theorem~\ref{perron}. Let us remark that we assume as main conditions~$s \in (0,1)$,~$p\in(1,\infty)$ and~$f$ satisfying conditions~\eqref{reg f},~\eqref{bound f} and~\eqref{monotonicity}, without recalling them every time in the statements of the results presented below.
           
      \begin{defn}[{\bf Perron solutions}]\label{perron classes and solutions} Let~$\Omega$ be an open and bounded subset of~$\h^n$ and assume that~$g \in W^{s,p}(\h^n)$. We define the \textup{upper class} ~$\mathcal{U}_g$ as the family of all functions~$u$ such that
      \begin{enumerate}[\rm(i)]
      \item $u: \h^n \rightarrow (-\infty,+\infty]$ is lower semicontinuous in~$\Omega$ and
      there exists~$h \in L^{p-1}_{sp}(\h^n)$ such that~$u \leq h$;
      \item  Given~$D \Subset \Omega$ and $v \in C(\overline{D})$ a weak solution in~$D$ to~\eqref{problema2} such that~$v \leq u$ almost everywhere in $\h^n \smallsetminus D$ and
             $$
            v(\xi) \leq \liminf\limits_{\eta \rightarrow \xi} u(\eta), \qquad \forall \xi \in \partial D.
             $$ 
             Then,~$v \leq u$ a.\!~e. in~$D$ as well;
      \item  $\liminf\limits_{\eta \rightarrow \xi \atop \eta \in \Omega}u(\eta) \geq \textup{ess}\limsup\limits_{\eta \rightarrow \xi \atop \eta \in \h^n \smallsetminus \Omega}g(\eta)$ for all~$\xi \in \partial \Omega$;
      \item $u \geq  g $ a.\!~e. in~$\h^n \smallsetminus \Omega$.
      \end{enumerate}
      The \textup{lower class} is~$\mathcal{L}_g:=\big\{ u: -u \in \mathcal{U}_{-g}\}$. The function~$\overline{H}_g := \inf\{u : u \in \mathcal{U}_g\}$ is the \textup{upper Perron solution} with boundary datum~$g$ in~$\Omega$ and~$\underline{H}_g := \sup\{u : u \in \mathcal{L}_g\}$ is the \textup{lower Perron solution} with boundary datum~$g$ in~$\Omega$.
      \end{defn}
     
     Some comments about the definition above must be made. First, let us note that by~\rm(i) it follows that~$|u| \neq \infty$ a.~\!e. in~$\h^n$, otherwise it can not be~$h \in L^{p-1}_{sp}(\h^n)$ such that~$u \leq h$ in~$\h^n$. Moreover, the upper class~$\mathcal{U}_g$ contains all lower semicontinuous weak supersolutions satisfying the boundary data.  Analogously the same holds true for weak subsolutions and for the lower class~$\mathcal{L}_g$. 
     
    In Definition~\ref{perron classes and solutions} we have assumed the boundary datum~$g$ in~$W^{s,p}(\h^n)$. In the lemma below we show that we can require less regularity on~$g$ in order to have a not empty upper class~$\mathcal{U}_g$.
     
     \begin{lemma}\label{s4 l1}
     If~$g \in L^{p-1}_{sp}(\h^n) \cap L^\infty(\h^n)$, then the upper class~$\mathcal{U}_g$ is not empty.
     \end{lemma}
     
     \begin{proof}
     Let us choose a constant~$M>0$ such that
     \begin{equation}\label{s4 l1 1}
     \sup_{\h^n} g +\c \, \|f\|_{L^\infty(\Omega)}^\frac{1}{p-1} \leq M < \infty,
     \end{equation}
     with the constant~$\c$ defined as follows
     $$
     \c^{-\frac{1}{p-1}} := 2 \, \min_{\xi \in \Omega} \int_{\h^n \smallsetminus B_r}\dd(\eta^{-1}\circ \xi)^{-Q-sp}\deta, \qquad \mbox{for } B_r: \Omega \Subset B_r.
     $$
     Note that in~\eqref{s4 l1 1} in the choice of the constant~$M$ we have also taken in consideration the non-homogeneity of the problem given by the presence of the supremum norm of the function~$f$.
     
     Take 
     $$
     u := M \, \chi_\Omega + g \, \chi_{\h^n \smallsetminus \Omega} \in W^{s,p}_{\rm loc}(\Omega)\cap L^{p-1}_{sp}(\h^n).
     $$ 
     Clearly Definition~\ref{perron classes and solutions}\rm(i) and~\rm(iii-iv) are satisfied.   
     In order to prove condition~\rm(ii) we show that~$u$ is a weak supersolution according to Definition~\ref{solution to inhomo pbm}. 
      
     Testing against a nonnegative test function~$\psi \in C^\infty_0(\Omega)$ gives
     \begin{align*}
    \int_{\h^n}\int_{\h^n} & L(u(\xi),u(\eta))(\psi(\xi)-\psi(\eta))\dd(\eta^{-1} \circ \xi)^{-Q-sp}\dxieta\\
     & = 2\int_{\h^n\smallsetminus \Omega}\int_{\Omega}  L(M,u(\eta))\psi(\xi)\dd(\eta^{-1} \circ \xi)^{-Q-sp}\dxieta\\
       & \geq  2\int_{\h^n\smallsetminus B_r}\int_{\Omega}  \c^{p-1} \, \|f\|_{L^\infty(\Omega)}\psi(\xi)\dd(\eta^{-1} \circ \xi)^{-Q-sp}\dxieta\\
       & \geq \int_{\Omega}f(\xi,u)\psi(\xi)\dxi.
     \end{align*}
      Hence,~$u$ is a weak supersolution in~$\Omega$ to~\eqref{problema2}. Thus, it also satisfies the comparison principle in~$\Omega$. Then,~$u \in \mathcal{U}_g$.
     \end{proof}
     
     \subsection{The Perron solutions}
     In this section we prove Theorem~\ref{perron}. 
     
     The first result is proving that our method is consistent. Indeed, let us show that if problem~\eqref{problema2} admits a weak solution~$h_g$ in~$\Omega$, then 
     $$
     \overline{H}_g = h_g = \underline{H}_g.
     $$
     \begin{lemma}\label{s4 l2}
     	Let us assume that~$h_g \in C(\overline{\Omega})$ is a weak solution in~$\Omega$ to~\eqref{problema2} such that
     	$$
     	\lim_{\eta \rightarrow \xi \atop \eta \in \Omega}h_g(\eta)=g(\xi) \ \mbox{for every} \ \xi \in \partial \Omega \qquad \mbox{and} \qquad h_g=g \ \mbox{a.~\!e. in } \h^n \smallsetminus \Omega,
     	$$
     	for some~$g \in C(\Omega') \cap L^{p-1}_{sp}(\h^n)$ with~$\Omega \Subset \Omega'$. Then,~$\overline{H}_g =h_g=\underline{H}_g$.
     \end{lemma}
     \begin{proof}
     	We prove the lemma for the upper solution~$\overline{H}_g$ being the other case symmetric. Clearly~$h_g$ belongs to~$\mathcal{U}_g$. Hence,~$h_g \geq \overline{H}_g$. To show that the reverse inequality holds true let~$u \in \mathcal{U}_g$. Then, for any~$\e>0$ there exists an open set~$D \Subset \Omega$ such that~$u+\e > h_g$ almost everywhere in~$\h^n \smallsetminus D$ and $h_g \leq \liminf_{\eta \rightarrow \xi}u(\eta) +\e$ for any $\xi \in \partial D$. Hence, by comparison principle we have that~$u +\e \geq h_g$ almost everywhere in~$D$ as well. Taking the infimum over~$\e>0$ and then over~$u \in \mathcal{U}_g$ we have that~$\overline{H}_g \geq h_g$ for a.~\!e.~$\xi \in \Omega$.  By the continuity
     	of $h_g$ and $\overline{H}_g$ (Proposition~\ref{s4 p2} below) we have $\overline{H}_g \geq h_g$ for any $\xi \in \Omega$.
     \end{proof}
     
     In order to prove Theorem~\ref{perron}, we have to introduce the {\it Poisson modification} of a function~$u$ in a set~$D$.

     \begin{defn}[{\bf Poisson modification of a continuous function}]\label{poisson mod}
     Let us consider a set~$D \Subset \Omega$ such that~$\h^n \smallsetminus D$ satisfies the measure density condition~\eqref{condizione misura} and let~$u \in C(\overline{D}) \cap L^{p-1}_{sp}(\h^n)$. We call \textup{Poisson modification} of~$u$ in the domain~$D$ the solution to the obstacle problem in~$D$ with boundary datum~$u$, i.~\!e.~$P_{u,D} \in \mathcal{K}_{u,-\infty}(D,\Omega) \cap C(\overline{D})$.
     \end{defn}
     
     Note that by Corollary~\ref{sec3.1_corol1}, we have that~$P_{u,D}$ is the weak solution of the following problem
     \begin{equation}\label{poisson modification}
     \begin{cases}
     \l w = f , \ & \mbox{in}\ D,\\
     w= u,  \ & \mbox{in}\ \h^n \smallsetminus D.
     \end{cases}
     \end{equation}
     
     By a density argument we extend the Definition~\ref{poisson mod} to all functions belonging to~$\mathcal{U}_g$.
     
     \begin{prop}\label{s4 p1}
     Let~$u \in \mathcal{U}_g$. Let~$D \Subset \Omega$ be an open subset of~$\Omega$ such that~$\h^n \smallsetminus D$ satisfies the measure density condition~\eqref{condizione misura}.
     Then, the Poisson modification of~$u$ in~$D$ satisfies 
     \begin{enumerate}[\rm(i)]
      \item $P_{u,D} \leq u$;
      \item $P_{u,D}$ is a continuous weak solution to~\eqref{problema2} in~$D$;
      \item $P_{u,D} \in \mathcal{U}_g$.
     \end{enumerate}
     \end{prop}
     
     \begin{proof}
      Applying the same procedure in~\cite[Lemma~9]{KKP17} we can construct an increasing sequence~$\{u_j\}$, $u_j \in C(\overline{D})$, of weak supersolutions to~\eqref{problema2} in~$D$, converging to~$u$ pointwise in~$\h^n$.
      Clearly by Definition~\ref{perron classes and solutions}\rm(i) we have~$u_j \leq u \leq h$, with~$h \in L^{p-1}_{sp}(\h^n)$.
      
      Define the Poisson modification of the function~$u_j$ in the set~$D$,~$P_{u_j,D} \in \mathcal{K}_{u_j,-\infty}(D,\Omega) \cap C(\overline{D})$ and denote with
      $$
      P_{u,D} := \lim_{j \rightarrow \infty}P_{u_j,D}.
      $$
     As consequence of the comparison principle in Proposition~\ref{comparison a.e.} and the fact that the sequence~$\{u_j\}$ is increasing we have that~$\{P_{u_j,D}\}$ is increasing as well and~$P_{u_j,D} \leq u$. 
     
     Hence, the pointwise limit
     $$
     P_{u,D}(\xi) = \lim_{j \rightarrow \infty}P_{u_j,D}(\xi),
     $$
     exists and satisfies~$P_{u,D} \leq u$. Moreover, the pointwise limit is lower semicontinuous. 
     
     Since the sequence~$\{P_{u_j,D}\}$ is increasing and bounded above by~$u$, we have that for any~$j \in \mathbb{N}$, $P_{u_1,D} \leq P_{u_j,D} \leq h$, with~$P_{u_1,D},h \in L^{p-1}_{sp}(\h^n)$. Thus, since is the limit of an increasing sequence of weak solutions in~$D$, by Corollary~\ref{corol sequence weak solutions}, we obtain that~$P_{u,D}$ is a continuous weak solution in~$D$. Hence, Proposition~\ref{s4 p1}\rm(i-ii) follow.
     
     We have to prove that~$P_{u,D} \in \mathcal{U}_g$. Note that we only have to prove Definition~\ref{perron classes and solutions}\rm(ii). Take an open set~$U \Subset \Omega$ and consider a weak solution~$v \in C(\overline{U})$ to problem~\eqref{problema2} such that~$v \leq P_{u,D}$ almost everywhere in~$\h^n \smallsetminus U$ and 
    $$
    v(\xi) \leq \liminf_{\eta \rightarrow \xi}P_{u,D}(\eta), \qquad \forall \xi \in \partial U.
    $$
    We claim that~$v \leq P_{u,D}$ almost everywhere in~$U$. Since~$P_{u,D} \leq u$ we have that~$v \leq u$ a.\!~e. in~$\h^n \smallsetminus U$ and 
     $$
    v(\xi) \leq \liminf_{\eta \rightarrow \xi}P_{u,D}(\eta) \leq \liminf_{\eta \rightarrow \xi}u(\eta), \qquad \forall \xi \in \partial U.
    $$
    Hence, $v \leq u$ almost everywhere in~$U$ as well. Now, if $U \cap D = \emptyset$ we are done since~$P_{u,D}\equiv u$ outside~$D$. Assume~$U \cap D \neq \emptyset$. We have that~$P_{u,D}$ is a weak solution in~$U \cap D$. In the set~$\h^n \smallsetminus U$, it holds $v \leq P_{u,D}$ almost everywhere and in~$U \cap (\h^n \smallsetminus D)$ we have that~$v \leq u \equiv P_{u,D}$ as shown above.
    
    Hence, in~$\h^n \smallsetminus (U \cap D)$ we have that~$v \leq P_{u,D}$ almost everywhere and on the boundary it holds
     $$
     \left\{
     \begin{array}{lr}
     v(\xi) \leq \liminf\limits_{\eta \rightarrow \xi \atop \eta \in  U \cap D}P_{u,D}(\eta), \qquad & \mbox{when}\	 \xi \in \partial U \cap D,\\
     v(\xi) \leq u(\xi) \leq \liminf\limits_{\eta \rightarrow \xi \atop \eta \in U \cap D }P_{u,D}(\eta), \qquad & \mbox{when}\	  \xi \in  U \cap \partial D.
     \end{array}
     \right.
    $$
   Hence by the comparison principle of Proposition~\ref{comparison a.e.} it follows that~$v \leq P_{u,D}$ a.\!~e. in~$U \cap D$ as well.
   \end{proof}
    
   Using Poisson modification and the H\"older continuity result of Theorem~\ref{teo_holder} by repeating the procedure used in~\cite{LL16}, we can prove the following proposition.
     
    \begin{prop}\label{s4 p2}
    The upper and lower Perron solutions~$\overline{H}_g$ and~$\underline{H}_g$ are continuous weak solution to~\eqref{problema2} in~$\Omega$, according to Definition~\ref{solution to inhomo pbm}.
    \end{prop}

  Finally we can prove our last main result.
  \begin{proof}[Proof of Theorem~{\rm \ref{perron}}]
  We divide the proof into two parts. Firstly we show that
  \begin{equation}\label{perron 1}
  \overline{H}_g \, \geq \, \underline{H}_g.
  \end{equation} 
  Let us choose~$\e>0$ and consider two function~$u \in \mathcal{U}_g$ and~$v \in \mathcal{L}_g$ such that
  $$
  \overline{H}_g + \e \geq u, \qquad \mbox{and} \qquad \underline{H}_g - \e \leq v.
  $$
  By Definition~\ref{perron classes and solutions}\rm(iii) we have that
  $$
  \liminf_{\eta \rightarrow \xi \atop \eta \in \Omega }\overline{H}_g(\eta)
   \geq \limsup_{\eta \rightarrow  \xi \atop \eta \in  \Omega}\underline{H}_g(\eta), \qquad \mbox{for all}~\xi \in \partial \Omega.
  $$
  Hence, since by Proposition~\ref{s4 p2}, both Perron solutions are weak solutions in~$\Omega$ to~\ref{problema2} and~$\overline{H}_g \geq \underline{H}_g$ a.~\!e. in~$\h^n \smallsetminus \Omega$,  we have that the desired inequality~\eqref{perron 1} follows by the comparison principle of Proposition~\ref{comparison a.e.}.
  
  We are left to prove the reverse inequality of~\eqref{perron 1} in~$\Omega$. With no loss of generality we may assume that the function~$g < \infty$ almost everywhere in~$\Omega$. Then, let us consider the solution to the obstacle problem~$u \in \mathcal{K}_{g,g}(\Omega,\Omega')$, with~$\Omega \Subset \Omega'$. Recalling that we assume the boundary datum~$g$ belonging to~$W^{s,p}(\h^n)$ we have that, by Theorem~\ref{t1 obst existence},~$u$ is a weak supersolution in~$\Omega$ and~$u \in W^{s,p}(\h^n)$. Hence,~$u \in \mathcal{U}_g$.

  Let us consider an exhaustion of~$\Omega$ with open sets~$\{D_j\}$:
  $$
  D_1 \, \subset \, D_2 \, \subset \, \dots \, \subset \, \Omega \, = \, \bigcup_{j=1}^\infty D_j,
  $$
  such that~$\h^n \smallsetminus D_j$ satisfies the measure density condition~\eqref{condizione misura}, for any~$j \in \mathbb{N}$, and take the sequence~$\{P_{u,D_j}\}$ of the Poisson modifications of~$u$ in~$\{D_j\}$.
  
  Clearly, by Proposition~\ref{s4 p1}, we have that~$P_{u,D_j}$ is a weak solution in~$D_j$, $P_{u,D_j} \equiv g$ in~$\h^n \smallsetminus \Omega$ and~$P_{u,D_j} \in \mathcal{U}_g$.
  
  Moreover,~$\{P_{u,D_j}\}$ is a non-increasing sequence by the comparison principle of Proposition~\ref{comparison a.e.} and
  $$
  u \geq P_{u,D_1} \geq P_{u,D_2} \geq \dots \atop
  \|u\|_{L^p(\h^n)} \geq \|P_{u,D_1}\|_{L^p(\h^n)} \geq \|P_{u,D_2}\|_{L^p(\h^n)} \geq \dots
  $$
  Hence, the pointwise limit~$P_{u,\Omega}=\lim_jP_{u,D_j}$ does exist almost everywhere and belongs to~$L^p(\h^n)$.
  
 Then, by in Definition~\ref{perron classes and solutions}\rm(i) we have that~$ P_{u,\Omega} \leq P_{u,D_j} \leq h $ with~$P_{u,\Omega},h \in L^{p-1}_{sp}(\h^n)$. By Corollary~\ref{corol sequence weak solutions} it follows that $P_{u,\Omega}$~is a weak solution in~$\Omega$ to~\eqref{problema2}. 
 
 Let us note that the function~$P_{u,\Omega}$ may not belong to the upper class~$\mathcal{U}_g$. Nevertheless, since~$P_{u,D_j} \in \mathcal{U}_g$, we have that for any~$\xi \in \h^n$
 $$
 \overline{H}_g (\xi) \leq P_{u,D_j}(\xi), \qquad   \overline{H}_g (\xi) \leq \lim_{j \rightarrow \infty}P_{u,D_j}(\xi) =P_{u,\Omega}(\xi).
 $$ 
 Thus, $\overline{H}_g \leq P_{u,\Omega}$.
 
 Repeating the same argument for the lower Perron solution we have that~$\underline{H}_g \geq P_{u,\Omega}$. Hence we have show that~$\overline{H}_g \leq \underline{H}_g$.
  
 We have that~$\overline{H}_g =\underline{H}_g =H_g$ and the proof ends applying Proposition~\ref{s4 p2}.
 \end{proof}

%
%


\vspace{5mm}

\begin{thebibliography}{99}
	
	\bibitem{AM18} {\sc Adimurthi, A. Mallick}:
	{A Hardy type inequality on fractional order Sobolev spaces on the Heisenberg group}.
	{\it Ann. Sc. Norm. Super. Pisa Cl. Sci. (5)} {\bf 18} (2018), no.~3, 917--949.   
	\vs
	
	\bibitem{Min03} {\sc B. Avelin, T. Kuusi, G. Mingione}: Nonlinear Calder\'on-Zygmund theory in the limiting case. {\it Arch. Rational Mech. Anal.}~{\bf 227}~(2018), 663--714.	
	\vs 
	
	\bibitem{BFS17} {\sc Z.~\!M. Balogh, K. F\"assler, H. Sobrino}: Isometric embeddings into Heisenberg groups. {\it Geom. Dedicata}~{\bf 195} (2017), no.~1, 163--192
	\vs
	
	\bibitem{BLU07} {\sc A. Bonfiglioli, E. Lanconelli, F. Uguzzoni}:
	{\it Stratified Lie Groups and their sub-Laplacians}.
	Springer Monographs in Mathematics, 2007. 
	\vs
	
	\bibitem{BDV20} {\sc C. Bucur, S. Dipierro, E. Valdinoci}: {On the mean value property of fractional harmonic functions}.
	\href{https://arxiv.org/abs/2002.09679}{\it arXiv:2002.09679}~(2020).
	\vs
	
	\bibitem{BS21} {\sc C. Bucur, M. Squassina}:
	An Asymptotic Expansion for the Fractional $p$-Laplacian and for Gradient Dependent Nonlocal Operators.
	{\it Commun. Contemp. Math.}~(2021). Available at \href{https://doi.org/10.1142/S0219199721500218}{https://doi.org/10.1142/S0219199721500218}
	\vs
	
	\bibitem{CS07} {\sc L. Caffarelli, L. Silvestre}: An extension problem related to the fractional Laplacian. 
	{\it Comm. Partial Differential Equations} {\bf 32}~(2007), 1245--1260.
	\vs
	
	\bibitem{DGP07}{\sc D. Danielli, N. Garofalo, A. Petrosyan}:
	The sub-elliptic obstacle problem:~$C^{1,\alpha}$ regularity of the free boundary in Carnot groups of step two.
	{\it Adv. Math.}~{\bf 211}~(2007), no.~2, 485--516.
	\vs
	
	\bibitem{DGS03}{\sc D. Danielli, N. Garofalo, S. Salsa}:
	Variational inequalities with lack of ellipticity. Part I: Optimal interior regularity and non-degeneracy of the free boundary.
   {\it Indiana Univ. Math. J.}~{\bf 52}~(2003), no.~2, 361--398.
   \vs
   
   \bibitem{DM20}{\sc C. De Filippis, G. Mingione}: A borderline case of Calder\'on-Zygmund estimates for nonuniformly elliptic problems. {\it St. Petersburg Math. J.} 31~(2020), no.~3, 455--477.
   \vs
   
  	\bibitem{DFP19} {\sc C. De Filippis, G. Palatucci}: H\"older regularity for nonlocal double phase equations. 
   {\it J. Differential Equations} {\bf 267} (2019), no.~1, 547--586.
   \vs
   
   \bibitem{DKP14} {\sc A. Di Castro, T. Kuusi, G. Palatucci}: Nonlocal Harnack inequalities. {\it J. Funct. Anal.}~{\bf 267}~(2014), no.~6, 1807--1836.
	\vs
	
	\bibitem{DKP16} {\sc A. Di Castro, T. Kuusi, G. Palatucci}: Local behavior of fractional $p$-minimizers.
	{\it Ann. Inst. H. Poincar\'e Anal. Non Lin\'eaire} {\bf  33} (2016), 1279--1299.
	\vs
	
	\bibitem{DPV12} {\sc E. Di Nezza, G. Palatucci, E. Valdinoci}: Hitchhiker's guide to the fractional Sobolev spaces. {\it Bull. Sci. Math.} {\bf 136}~(2012), 521--573.
	\vs
	
	\bibitem{FF15} {\sc F. Ferrari, B. Franchi}:
	{Harnack inequality for fractional Laplacians in Carnot groups}.
	{\it Math. Z.} {\bf 279}~(2015), 435--458.
	\vs	
     
     \bibitem{FMPPS18} {\sc F. Ferrari, M. Miranda Jr, D. Pallara, A. Pinamonti, Y. Sire}: Fractional Laplacians, perimeters and heat semigroups in Carnot groups,
     {\it Discrete Cont. Dyn. Sys - Series S}\, {\bf 11}~(2018), 477-491.
     \vspace{-2.5mm}
	    
    \bibitem{GLM86} {\sc S. Granlund, P. Lindqvist, O. Martio}: Note on the PWB-method in the nonlinear case. {\it Pacific J. Math.}~{ \bf 125}~(1986), 381–395
    \vs
    
	\bibitem{HKM06} {\sc J. Heinonen, T. Kilpel\"ainen, O. Martio}:
	{\it Nonlinear Potential Theory of Degenerate Elliptic Equations}. 
	Dover Publications Inc., Mineola, New York, 2006.
	\vs
	
	\bibitem{KS18} {\sc A. Kassymov, D. Surgan}:
	{Some functional inequalities for the fractional $p$-sub-Laplacian}. \href{https://arxiv.org/abs/1804.01415v2}{\it arXiv:1804.01415}~(2018).    
	\vs
	
	\bibitem{KS21} {\sc A. Kassymov, D. Surgan}: Existence of solutions for p-sub-Laplacians with nonlinear sources on the Heisenberg group,
	{\it Complex Var. Elliptic Equ.}. {\bf 66}~(2021),
     614--625.
    \vs

	
	\bibitem{KKP16} {\sc J. Korvenp\"a\"a, T. Kuusi, G. Palatucci}: The obstacle problem for nonlinear integro-differential operators. 
	{\it Calc. Var. Partial Differential Equations} {\bf 55} (2016), no.~3, Art.~63.
	\vs
		
	\bibitem{KKP17} {\sc J. Korvenp\"a\"a, T. Kuusi, G. Palatucci}: Fractional superharmonic functions and the Perron method for nonlinear integro-differential equations. {\it Math. Ann.} {\bf 369} (2017), no.~3-4, 1443--1489.\!\!\!
	\vs
	
	\bibitem{KMS15} {\sc T. Kuusi, G. Mingione,  Y. Sire}: Nonlocal self-improving properties. {\it Anal. PDE} {\bf 8} (2015), no.~1, 57--114.
	\vs	
     \bibitem{LL16} {\sc E. Lindgren, P. Lindqvist}:
     Perron's Method and Wiener's Theorem for a Nonlocal Equation, {\it Potential Anal.} {\bf 46}~(2017), no.~4, 705--737
     \vs

 
     \bibitem{MPPP21} {\sc M. Manfredini, G. Palatucci, M. Piccinini, S. Polidoro}:
     {H\"older continuity and boundedness estimates for nonlinear fractional equations in the  Heisenberg group}. {\it J. Geom. Anal.}~{\bf33}, 77~(2023).
     \vs  

	\bibitem{Mou17} {\sc C. Mou}: Perron's method for nonlocal fully nonlinear equations, {\it Anal. PDE}~{\bf 10}~(2017), 1227--1254.
	\vs	
		
	\bibitem{Pal18}{\sc G. Palatucci}: The Dirichlet problem for the $p$-fractional Laplace equation. 
	{\it Nonlinear Anal.} {\bf 177} (2018), 699--732.
	\vs

	\bibitem{PP21} {\sc G. Palatucci, M. Piccinini}:
	{Nonlocal Harnack inequalities in the Heisenberg group}.
	{\it Calc. Var. Partial Differential Equations}~{\bf61}, 185~(2022).
	\vs
	
	    \bibitem{PSV13} {\sc G. Palatucci, O. Savin, E. Valdinoci}: Local and Global minimizers for a variational energy involving a fractional norm. {\it Ann. Mat. Pura Appl.} {\bf 192}~(2013), no.~4, 673--718.\!
	\vs
	
	\bibitem{PV13} {\sc A. Pinamonti, E. Valdinoci}: A Lewy-Stampacchia estimate for variational inequalities in the Heisenberg group. {\it Rend.
     Istit. Mat. Univ. Trieste}~{\bf 45} (2013), 23--45.
      \vs

	\bibitem{SM20} {\sc J.~M. Scott, T. Mengesha}:
	{Self-improving Inequalities for bounded weak solutions to nonlocal double phase equations}.
    {\it Commun. Pure Appl. Anal.} {\bf 21}~(2022), no.~1, 183--212.
	\vs
	
	\bibitem{WD20} {\sc X. Wang, G. Du}:
	{Properties of solutions to fractional $p$-subLaplace equations on the Heisenberg group}.
	{\it  Bound. Value Probl.} (2020), Art.~128.
	\vs
\end{thebibliography}
\end{document}